\documentclass[aos,preprint]{imsart}
\RequirePackage[OT1]{fontenc}
\RequirePackage{amsmath, amsfonts, amssymb, amsthm, amsbsy, mathrsfs, amscd, bm}
\RequirePackage[numbers]{natbib}
\RequirePackage[colorlinks,citecolor=blue,urlcolor=blue]{hyperref}
\usepackage{url}
\usepackage{graphicx}
\usepackage{xspace} 
\usepackage[small,bf]{caption}
\usepackage{subcaption}
\usepackage{color}
\usepackage[normalem]{ulem}
\usepackage{algorithm, algorithmic}
\usepackage{relsize}

\def\diff{\mathrm{d}}                 

\def\1i{\imath}                     
\def\iid{\textrm{i.i.d.}\xspace}                

\def\argmin{\mathop{\hbox{argmin}}}

\arxiv{arXiv:1511.01957}

\startlocaldefs
\numberwithin{equation}{section}
\theoremstyle{plain}
\newtheorem{theorem}{Theorem}[section]
\newtheorem{lemma}[theorem]{Lemma}

\theoremstyle{remark}
\newtheorem{remark}{Remark}

\newcommand{\R}{\mathbb{R}}

\newcommand{\goto}{\rightarrow}
\newcommand{\sgn}{\mathop\mathrm{sgn}}
\renewcommand{\P}{\operatorname{\mathbb{P}}}
\newcommand{\E}{\operatorname{\mathbb{E}}}

\renewcommand{\Theta}{\Pi}

\newcommand{\e}{\mathrm{e}}

\newcommand{\transp}{\top}

\newcommand{\supp}[1]{\operatorname{supp}(#1)}

\newcommand{\half}{{\textstyle\frac{1}{2}}}

\newcommand{\bX}{\bm{X}}

\newcommand{\bb}{\bm{b}}
\newcommand{\by}{\bm{y}}

\newcommand{\bz}{\bm{z}}

\newcommand{\bbeta}{\bm{\beta}}
\newcommand{\bg}{\bm{g}}

\newcommand{\fdp}{\textnormal{FDP}}

\newcommand{\power}{\textnormal{TPP}}

\newcommand{\mfdr}{q^{\star}}

\newcommand{\plim}{\xrightarrow{ ~ \mathbb{P} ~}}

\newcommand{\ffdp}{\textnormal{fdp}^{\mathsmaller \infty}}
\newcommand{\fpower}{\textnormal{tpp}^{\mathsmaller \infty}}
\newcommand{\ftpp}{\textnormal{tpp}^{\mathsmaller \infty}}
\newcommand{\fd}{\textnormal{fd}^{\mathsmaller\infty}}
\newcommand{\td}{\textnormal{td}^{\mathsmaller\infty}}
\newcommand{\T}{\top}

\graphicspath{{figs/}}

\endlocaldefs

\begin{document}

\begin{frontmatter}
\title{False Discoveries Occur Early on the Lasso Path}
\runtitle{False Discoveries on Lasso Path}

\begin{aug}
\author{\fnms{Weijie} \snm{Su}\thanksref{t1}\ead[label=e1]{suw@wharton.upenn.edu}}
\and
\author{\fnms{Ma{\l}gorzata} \snm{Bogdan}\thanksref{t2}\ead[label=e2]{malgorzata.bogdan@pwr.edu.pl}}
\and
\author{\fnms{Emmanuel} \snm{Cand\`es}\thanksref{t3}\ead[label=e3]{candes@stanford.edu}}

\thankstext{t1}{Supported in part by NSF under grant CCF-0963835.}
\thankstext{t2}{Supported in part by the European Union's 7th
Framework Programme for research, technological development and
demonstration under Grant Agreement No.~602552 and co-financed by the
Polish Ministry of Science and Higher Education under Grant Agreement
2932/7.PR/2013/2}
\thankstext{t3}{Supported in part by NSF under grant CCF-0963835 and by the Math + X Award from the Simons Foundation.}
\runauthor{W.~Su, M.~Bogdan, and E.~Cand\`es}

\affiliation{University of Pennsylvania, University of Wroclaw, and Stanford University}

\address{Departments of Statistics\\
The Wharton School\\
University of Pennsylvania\\
Philadelphia, PA 19104\\
USA\\
\printead{e1}}


\address{Institute of Mathematics\\
University of Wroclaw\\
50-137, Wroclaw\\
Poland\\
\printead{e2}}

\address{Departments of Statistics and Mathematics\\
Stanford University\\
Stanford, CA 94305\\
USA\\
\printead{e3}}
\end{aug}

\begin{abstract}
  In regression settings where explanatory variables have very low
  correlations and there are relatively few effects, each of large
  magnitude, we expect the Lasso to find the important variables with
  few errors, if any. This paper shows that in a regime of linear
  sparsity---meaning that the fraction of variables with a
  non-vanishing effect tends to a constant, however small---this cannot
  really be the case, even when the design variables are
  stochastically independent. We demonstrate that true features and
  null features are always interspersed on the Lasso path, and that
  this phenomenon occurs no matter how strong the effect sizes are. We
  derive a sharp asymptotic trade-off between false and true positive
  rates or, equivalently, between measures of type I and type II
  errors along the Lasso path. This trade-off states that if we ever
  want to achieve a type II error (false negative rate) under a
  critical value, then anywhere on the Lasso path the type I error
  (false positive rate) will need to exceed a given threshold so that
  we can never have both errors at a low level at the same time. Our
  analysis uses tools from approximate message passing (AMP) theory as
  well as novel elements to deal with a possibly adaptive selection of
  the Lasso regularizing parameter.
\end{abstract}

\begin{keyword}[class=MSC]
\kwd[Primary ]{62F03}
\kwd[; secondary ]{62J07}
\kwd{62J05}.
\end{keyword}

\begin{keyword}
\kwd{Lasso}
\kwd{Lasso path}
\kwd{false discovery rate}
\kwd{false negative rate}
\kwd{power}
\kwd{approximate message passing (AMP)}
\kwd{adaptive selection of parameters}
\end{keyword}

\end{frontmatter}

\section{Introduction}\label{sec:introduction}

Almost all data scientists know about and routinely use the Lasso
\cite{santosa1986linear,Lasso} to fit regression models. In the big
data era, where the number $p$ of explanatory variables often exceeds
the number $n$ of observational units, it may even supersede the
method of least-squares. One appealing feature of the Lasso over
earlier techniques such as ridge regression is that it automatically
performs variable reduction, since it produces models where lots
of---if not most---regression coefficients are estimated to be exactly
zero. In high-dimensional problems where $p$ is either comparable to
$n$ or even much larger, the Lasso is believed to select those
important variables out of a sea of potentially many irrelevant
features.

Imagine we have an $n \times p$ design matrix $\bX$ of features, and
an $n$-dimensional response $\by$ obeying the standard linear model
\[
\by = \bX \bbeta+ \bz,
\]
where $\bz$ is a noise term. The Lasso is the solution to
\begin{equation} 
\label{eq:Lasso}
\widehat \bbeta(\lambda) = \argmin_{\bm b \in \R^p} \,\,\, \half \| \by - \bX \bb\|^2 + \lambda \,  \|\bb\|_1;
\end{equation}
if we think of the noise term as being Gaussian, we interpret it as a
penalized maximum likelihood estimate, in which the fitted
coefficients are penalized in an $\ell_1$ sense, thereby encouraging
sparsity. (There are nowadays many variants on this idea including
$\ell_1$-penalized logistic regression \cite{Lasso}, elastic nets
\cite{el_net}, graphical Lasso \cite{gLasso}, adaptive Lasso
\cite{zou2006adaptive}, and many others.) As is clear from
\eqref{eq:Lasso}, the Lasso depends upon a regularizing parameter
$\lambda$, which must be chosen in some fashion: in a great number of
applications this is typically done via adaptive or data-driven
methods; for instance, by cross-validation
\cite{pnasLasso,natureLasso1,natureLasso2,scienceLasso}. Below, we
will refer to the Lasso path as the family of solutions
$\widehat\bbeta(\lambda)$ as $\lambda$ varies between $0$ and
$\infty$. We say that a variable $j$ is selected at $\lambda$ if
$\widehat{\beta}_j(\lambda) \neq 0$.\footnote{We also say that a
  variable $j$ enters the Lasso path at $\lambda_0$ if there is there
  is $\varepsilon > 0$ such that $\widehat{\beta}_j(\lambda) = 0$ for
  $\lambda \in [\lambda_0 - \varepsilon, \lambda_0]$ and
  $\widehat{\beta}_j(\lambda) \neq 0$ for
  $\lambda \in (\lambda_0, \lambda_0 + \varepsilon]$. Similarly a
  variable is dropped at $\lambda_0$ if
  $\widehat{\beta}_j(\lambda) \neq 0$ for
  $\lambda \in [\lambda_0 - \varepsilon, \lambda_0)$ and
  $\widehat{\beta}_j(\lambda) =0$ for
  $\lambda \in [\lambda_0, \lambda_0 + \varepsilon]$.}

The Lasso is, of course, mostly used in situations where the true
regression coefficient sequence is suspected to be sparse or nearly
sparse. In such settings, researchers often believe---or, at least,
wish---that as long as the true signals (the nonzero regression
coefficients) are sufficiently strong compared to the noise level and
the regressor variables weakly correlated, the Lasso with a carefully
tuned value of $\lambda$ will select most of the true signals while
picking out very few, if any, noise variables.  This belief is
supported by theoretical asymptotic results discussed below, which
provide conditions for perfect support recovery, i.e.~for perfectly
identifying which variables have a non-zero effect, see
\cite{lassogaussian,wainwright2009information,reeves2013} for
instance. Since these results guarantee that the Lasso works well in
an extreme asymptotic regime, it is tempting to over-interpret what
they actually say, and think that the Lasso will behave correctly in
regimes of practical interest and offer some guarantees there as well.
However, some recent works such as \cite{fan2010sure} have observed
that the Lasso has problems in selecting the proper model in practical
applications, and that false discoveries may appear very early on the
Lasso path. This is the reason why
\cite{Buhlmann_book,Buhlmann_discussion,pokarowski} suggest that the
Lasso should merely be considered as a {\em variable screener} rather
than a {\em model selector}.

While the problems with the Lasso ordering of predictor variables are
recognized, they are often attributed to (1) correlations between
predictor variables, and (2) small effect sizes.  In contrast, the
novelty and message of our paper is that the selection problem also
occurs when the signal-to-noise ratio is infinitely large (no noise)
and the regressors are stochastically independent; we
  consider a random design $\bm X$ with independent columns, and as a result, all population correlations vanish (so the sample correlations are small). We also explain that this phenomenon is mainly due to
  the shrinkage of regression coefficients, and does not occur when
  using other methods, e.g.~an $\ell_0$ penalty in \eqref{eq:Lasso}
  rather than the $\ell_1$ norm, compare Theorem \ref{thm:l0} below.

Formally, we study the value of the false discovery proportion (FDP),
the ratio between the number of false discoveries and the total number
of discoveries, along the Lasso path.\footnote{Similarly, the TPP is
  defined as the ratio between the number of true discoveries and that
  of potential true discoveries to be made.} This requires notions of
true/false discoveries, and we pause to discuss this important point.
In high dimensions, it is not a trivial task to define what are true
and false discoveries, see
e.g.~\cite{Berk_2013,Max_False,Tibshirani_post_selection_2015,HD_Knockoffs,Taylor_post_selection_2016};
these works are concerned with a large number of correlated
regressors, where it is not clear which of these should be selected in
a model. In response, we have selected to work in the very special
case of {\em independent} regressors precisely to analyze a context
where such complications do not arise and it is, instead, quite clear
what true and false discoveries are. We classify a selected regressor
$\bm X_j$ to be a false discovery if it is stochastically independent from
the response, which in our setting is equivalent to $\beta_j = 0$.
Indeed, under no circumstance can we say that that such a variable,
which has zero explanatory power, is a true discovery.


 Having clarified this point
and as a setup for our theoretical findings, Figure~\ref{fig:compare}
studies the performance of the Lasso under a $1010 \times 1000$ a
random Gaussian design, where the entries of $\bm X$ are independent
draws from $\mathcal{N} (0,1)$.  Set
$\beta_1 = \cdots = \beta_{200} = 4$,
$\beta_{201} = \cdots = \beta_{1000} = 0$ and the errors to be
independent standard normals. Hence, we have 200 nonzero coefficients
out of 1000 (a relatively sparse setting), and a very large
signal-to-noise ratio (SNR). For instance, if we order the variables
by the magnitude of the least-squares estimate, which we can run since
$n =1010 > 1000=p$, then with probability practically equal to one, all
the top 200 least-squares discoveries correspond to true discoveries,
i.e.~variables for which $\beta_j = 4$.  This is in sharp contrast
with the Lasso, which selects null variables rather early. To be sure,
when the Lasso includes half of the true predictors so that the false
negative proportion falls below 50\% or true positive proportion (TPP)
passes the 50\% mark, the FDP has already passed 8\% meaning that we
have already made 9 false discoveries. The FDP further increases to
19\% the first time the Lasso model includes all true predictors,
i.e.~achieves full power (false negative proportion vanishes).
\begin{figure}[htp!]
\centering
\includegraphics[scale=0.7]{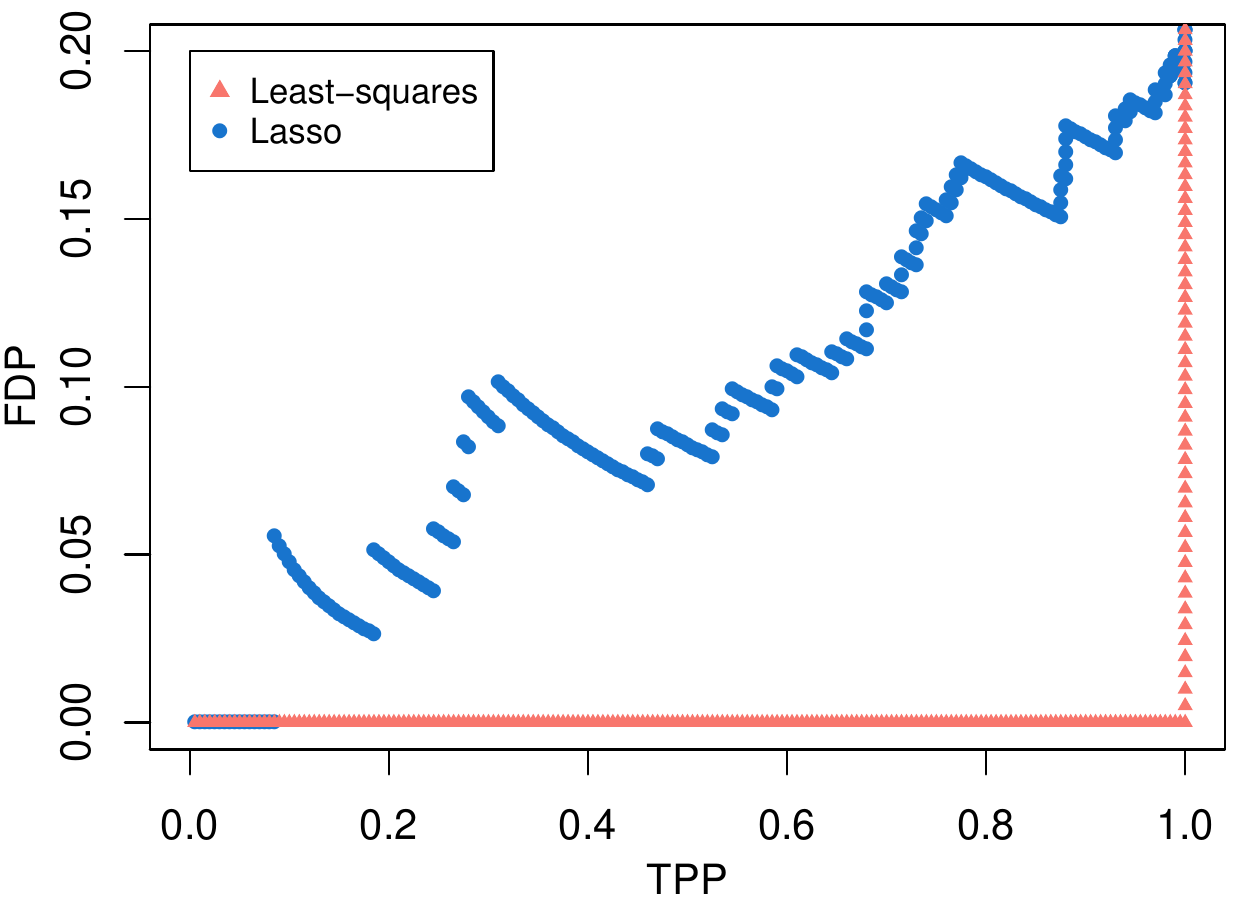}
\caption{True positive and false positive rates along the Lasso path
  {as compared to the ordering provided by the least-squares estimate}.}
\label{fig:compare}
\end{figure}

Figure~\ref{fig:intro_hist} provides a closer look at this phenomenon,
and summarizes the outcomes from 100 independent experiments under the
same Gaussian random design setting. In all the simulations, the first
noise variable enters the Lasso model before 44\% of the true signals
are detected, and the last true signal is preceded by at least 22 and,
sometimes, even 54 false discoveries. On average, the Lasso detects
about 32 signals before the first false variable enters; to put it
differently, the TPP is only 16\% at the time the first false
discovery is made. The average FDP evaluated the first time all
signals are detected is 15\%. For related empirical results, see
e.g.~\cite{fan2010sure}.
\begin{figure}[h!]
\centering
\includegraphics[scale=0.48]{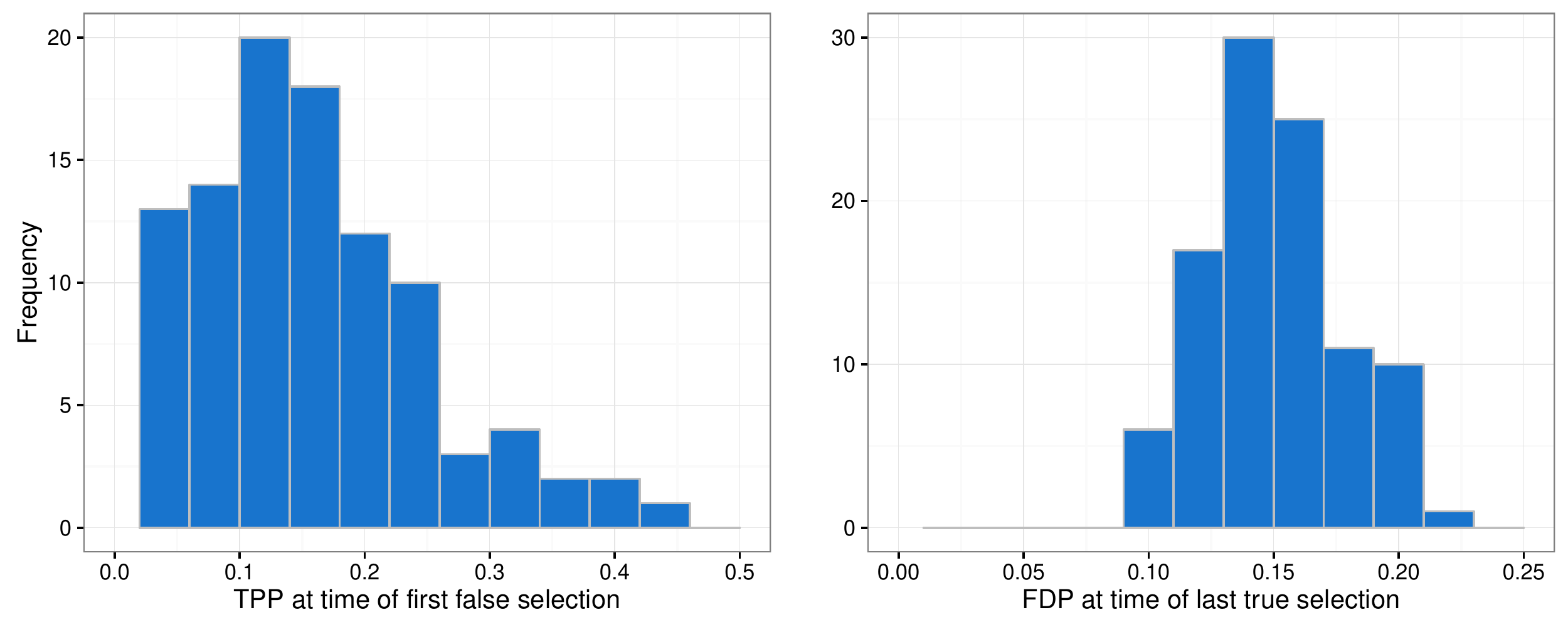}
\caption{Left: power when the first false variable enters the Lasso
  model. Right: false discovery proportion the first time power
  reaches one (false negative proportion vanishes).}
\label{fig:intro_hist}
\end{figure}


The main contribution of this paper is to provide a quantitative
description of this phenomenon in the asymptotic framework of {\em
  linear sparsity} defined below and previously studied e.g.~in
\cite{BM12}. Assuming a random design with independent Gaussian
predictors as above, we derive a fundamental Lasso trade-off between
power (the ability to detect signals) and type I errors or, said
differently, between the true positive and the false positive
rates. This trade-off says that it is impossible to achieve high power
and a low false positive rate simultaneously. Formally, we compute the
formula for an exact boundary curve separating achievable
$(\power, \fdp)$ pairs from pairs that are impossible to achieve no
matter the value of the signal-to-noise ratio (SNR). Hence, we prove
that there is a whole favorable region in the $(\power,\fdp)$ plane
that cannot be reached, see Figure~\ref{fig:phase} for an
illustration.

\section{The Lasso Trade-off Diagram}
\label{sec:main-results}

\subsection{Linear sparsity and the working model}

 We mostly work in the setting of \cite{BM12}, which
  specifies the design $\bX \in \R^{n \times p}$, the parameter
  sequence $\bbeta \in \R^p$ and the errors $\bm z \in \R^n$. The
  design matrix $\bX$ has \iid~$\mathcal{N}(0, 1/n)$ entries so that
  the columns are approximately normalized, and the errors $z_i$ are
  \iid~$\mathcal{N}(0, \sigma^2)$, where $\sigma$ is fixed but
  otherwise arbitrary. Note that we do not exclude the value
  $\sigma = 0$ corresponding to noiseless observations. The regression
  coefficients $\beta_1, \ldots, \beta_p$ are independent copies of a
  random variable $\Pi$ obeying $\E \Pi^2 < \infty$ and
  $\P(\Theta \ne 0) = \epsilon \in (0,1)$ for some constant
  $\epsilon$.  For completeness, $\bX, \bbeta$, and $\bz$ are all
  independent from each other. As in \cite{BM12}, we are interested in
  the limiting case where $p, n \goto \infty$ with
  $n/p \rightarrow \delta$ for some positive constant $\delta$.  A few
  comments are in order.

  \paragraph{Linear sparsity}  The first concerns the degree of
  sparsity. In our model, the expected number of nonzero regression
  coefficients is linear in $p$ and equal to $\epsilon \cdot p$ for
  some $\epsilon > 0$. Hence, this model excludes a form of
  asymptotics discussed in
  \cite{lassogaussian,wainwright2009information,reeves2013}, for
  instance, where the fraction of nonzero coefficients vanishes in the
  limit of large problem sizes. Specifically, our results do not
  contradict asymptotic results from \cite{lassogaussian} predicting
  perfect support recovery in an asymptotic regime, where the number
  of $k$ of variables in the model obeys
  $k/p \le \delta/(2\log p) \cdot (1 + o(1))$ and the effect sizes all
  grow like $c \cdot \sigma \sqrt{2\log p}$, where $c$ is an unknown
  numerical constant.  The merit of the linear sparsity regime lies in
  the fact that our theory makes accurate predictions when describing
  the performance of the Lasso in practical settings with moderately
  large dimensions and reasonable values of the degree of sparsity,
  including rather sparse signals. The precision of these predictions
  is illustrated in Figure \ref{fig:simu_FDP_TPP} and in Section
  \ref{sec:discussion}. In the latter case, $n=250$, $p=1000$ and the
  number of $k$ of signals is very small, i.e.~$k = 18$.


\paragraph{Gaussian designs} Second, Gaussian designs with independent
columns are believed to be ``easy'' or favorable for model selection
due to weak correlations between distinct features. (Such designs
happen to obey restricted isometry properties \cite{CanTaoDecode05} or
restricted eigenvalue conditions \cite{bickel2009} with high
probability, which have been shown to be useful in settings sparser
than those considered in this paper.) Hence, negative results under
the working hypothesis are likely to extend more generally.

\paragraph{Regression coefficients} Third, the assumption concerning
the distribution of the regression coefficients can be slightly
weakened: all we need is that the sequence $\beta_1, \ldots, \beta_p$
has a convergent empirical distribution with bounded second moment.
We shall not pursue this generalization here.

\subsection{Main result}

Throughout the paper, $V$ (resp.~$T$) denotes the number of Lasso
false (resp.~true) discoveries while $k = |\{j: \beta_j \ne 0\}|$
denotes the number of true signals; formally,
$V(\lambda) = |\{j: \widehat\beta_j(\lambda) \ne 0 \text{ and } \beta_j =
0\}|$
whereas
$T(\lambda) = |\{j: \widehat\beta_j(\lambda) \ne 0 \text{ and } \beta_j
\ne 0\}|$.  With this, we define the FDP as usual,
\begin{equation}
\label{eq:fdp}
\fdp(\lambda) = \frac{V(\lambda)}{|\{j: \widehat\beta_j(\lambda) \neq 0\}| \vee 1}
\end{equation}
and, similarly, the TPP is defined as
\begin{equation}
\label{eq:tpp}
\power(\lambda) = \frac{T(\lambda)}{k \vee 1}
\end{equation}
(above, $a \vee b = \max \{a, b\}$). The dependency on $\lambda$ shall
often be suppressed when clear from the context.  Our main result
provides an explicit trade-off between FDP and TPP.
\begin{theorem}\label{thm:min_fdr_given_power}
  Fix $\delta \in (0, \infty)$ and $\epsilon \in (0, 1)$, and consider
  the function $\mfdr(\cdot) = \mfdr(\cdot; \delta, \epsilon) > 0$
  given in \eqref{eq:q}. Then under the working hypothesis and for any
  arbitrary small constants $\lambda_0 >0$ and $\eta > 0$, the
  following conclusions hold:
\begin{enumerate}
\item[(a)] In the {\bf noiseless case} ($\sigma = 0$), the event
\begin{equation} 
\label{eq:main}
\bigcap_{\lambda \ge \lambda_0} \Big\{ \fdp(\lambda) \ge \mfdr\left( \power(\lambda) \right) - \eta\Big\}
\end{equation}
holds with probability tending to one. (The lower bound on $\lambda$
in \eqref{eq:main} does not impede interpretability since we are not
interested in variables entering the path last.) 

\item[(b)] With {\bf noisy data} ($\sigma > 0$) the conclusion is
  exactly the same as in (a).

\item[(c)] Therefore, in both the noiseless and noisy cases, no matter
  how we choose $\widehat\lambda(\by, \bX) \ge c_1$ {\bf adaptively} by
  looking at the response $\by$ and design $\bX$, with probability
  tending to one we will  never have
  $\fdp(\widehat{\lambda}) < q^\star(\power(\widehat{\lambda})) - c_2$. 

\item[(d)] The boundary curve $q^\star$ is {\bf tight}:  any continuous
  curve $q(u) \ge q^\star(u)$ with strict inequality for some $u$ will
  fail (a) and (b) for some prior distribution $\Pi$ on the regression
  coefficients.
\end{enumerate}
\end{theorem}

 A different way to phrase the trade-off is via false
  discovery and false negative rates. Here, the FDP is a natural
  measure of type I error while $1 - \power$ (often called the false
  negative proportion) is the fraction of missed signals, a natural
  notion of type II error. In this language, our results simply say
  that nowhere on the Lasso path can both types of error rates be
  simultaneously low. 


\begin{remark}\label{rem:fixed_prior}
We would like to emphasize that the boundary is derived from a
best-case point of view. For a fixed prior $\Pi$, we also provide in
Theorem \ref{thm:instance_specific} from Appendix
\ref{sec:proof-theor-refthm:m} a trade-off curve $q^\Pi$ between TPP
and FDP, which always lies above the boundary $q^\star$. Hence, the
trade-off is of course less favorable when we deal with a specific
Lasso problem.  In fact, $q^\star$ is nothing else but the lower
envelope of all the instance-specific curves $q^\Pi$ with $\P(\Pi \neq
0) = \epsilon$.
\end{remark}

{Figure~\ref{fig:phase} presents two instances of the {\em Lasso
    trade-off diagram}, where the curve $q^\star(\cdot)$ separates the
  red region, where both type I and type II errors are small, from the
  rest (the white region). Looking at this picture, Theorem
  \ref{thm:min_fdr_given_power} says that nowhere on the Lasso path we
  will find ourselves in the red region, and that this statement
  continues to hold true even when there is no noise. Our theorem also
  says that we cannot move the boundary upward. As we shall see, we
  can come arbitrarily close to any point on the curve by specifying a
  prior $\Pi$ and a value of $\lambda$. Note that the right plot is
  vertically truncated at $0.6791$, implying that TPP cannot even
  approach 1 in the regime of $\delta = 0.3, \epsilon=0.15$. This
  upper limit is where the Donoho-Tanner phase transition occurs
  \cite{DonTan05}, see the discussion in Section \ref{sec:techn-contr}
  and Appendix~\ref{sec:optimizing-tradeoff}.}
\begin{figure}[!htp]
\centering
\includegraphics[scale=0.37]{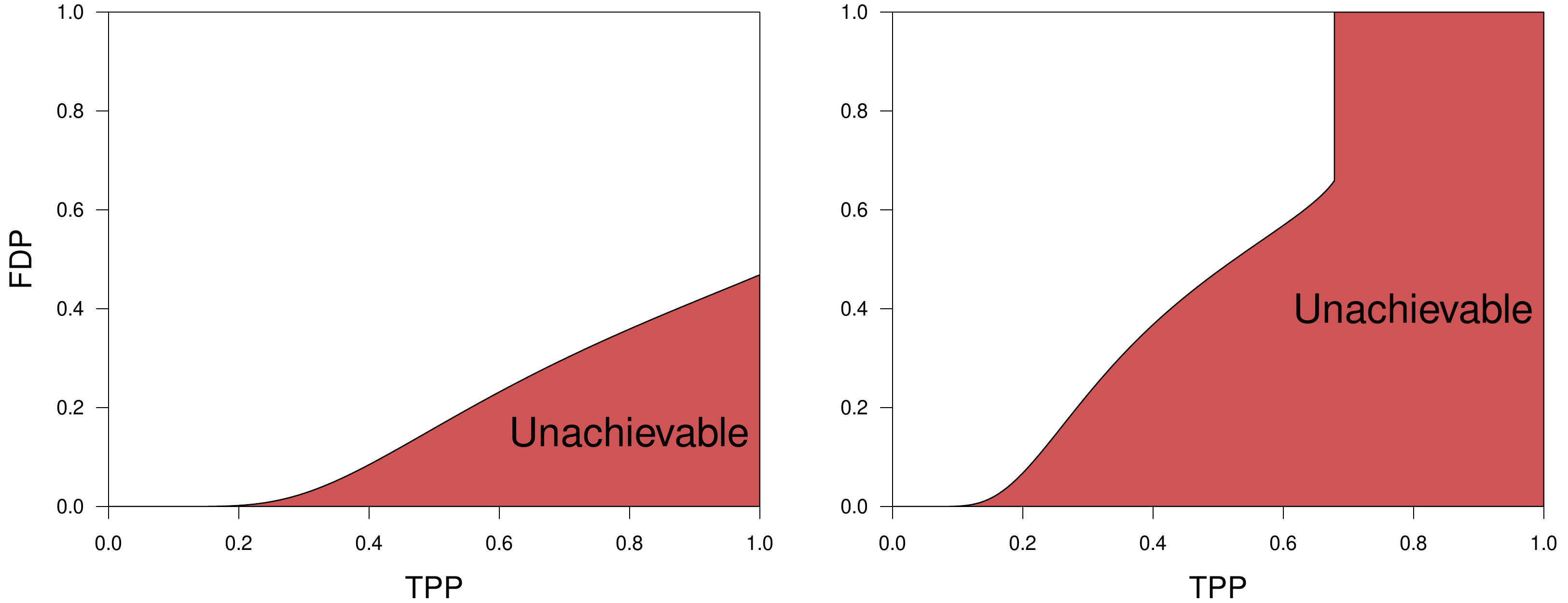}
\caption{The Lasso trade-off diagram: left is with $\delta = 0.5$ and $\epsilon = 0.15$, and right is with $\delta = 0.3$ and $\epsilon=0.15$ (the vertical truncation occurs at $0.6791$). }
\label{fig:phase}
\end{figure}

Support recovery from noiseless data is presumably the most ideal
scenario. Yet, the trade-off remains the same as seen in the first
claim of the theorem.  As explained in
Section~\ref{sec:matter-with-ell_1}, this can be understood by
considering that the root cause underlying the trade-off in both the
noiseless and noisy cases come from the pseudo-noise introduced by
shrinkage.

\subsection{The boundary curve $q^\star$}

We now turn to specify $q^\star$. For a fixed $u$, let $t^\star(u)$ be
the largest positive root\footnote{If $u = 0$, treat $+\infty$ as a
  root of the equation, and in \eqref{eq:q} conventionally set
  $0/0=0$. In the case where $\delta \ge 1$, or $\delta < 1$ and
  $\epsilon$ is no larger than a threshold determined only by
  $\delta$, the range of $u$ is the unit interval $[0, 1]$.
  Otherwise, the range of $u$ is the interval with endpoints 0 and
  some number strictly smaller than 1, see the discussion in
  Appendix~\ref{sec:optimizing-tradeoff}.} of the equation in $t$,
\[
\frac{2(1 - \epsilon)\left[ (1+t^2)\Phi(-t) - t\phi(t) \right] + \epsilon(1 + t^2) - \delta}{\epsilon\left[ (1+t^2)(1-2\Phi(-t)) + 2t\phi(t) \right]} = \frac{1 - u}{1 - 2\Phi(-t)}.
\]
Then
\begin{equation}\label{eq:q}
\mfdr(u; \delta, \epsilon) = \frac{2(1-\epsilon)\Phi(-t^\star(u))}{2(1-\epsilon)\Phi(-t^\star(u)) + \epsilon u}.
\end{equation}
It can be shown that this function is infinitely many times
differentiable over its domain, always strictly increasing, and
vanishes at $u = 0$. Matlab code to calculate $\mfdr$ is
available at \url{https://github.com/wjsu/fdrlasso}.

Figure~\ref{fig:qfunction} displays examples of the function
$q^\star$ for different values of $\epsilon$ (sparsity), and
$\delta$ (dimensionality). It can be observed that the issue of FDR
control becomes more severe when the sparsity ratio $\epsilon=k/p$
increases and the dimensionality $1/\delta = p/n$ increases. 
\begin{figure}[!htp]
\centering
\begin{subfigure}[b]{\textwidth}
\centering
\includegraphics[scale=0.39]{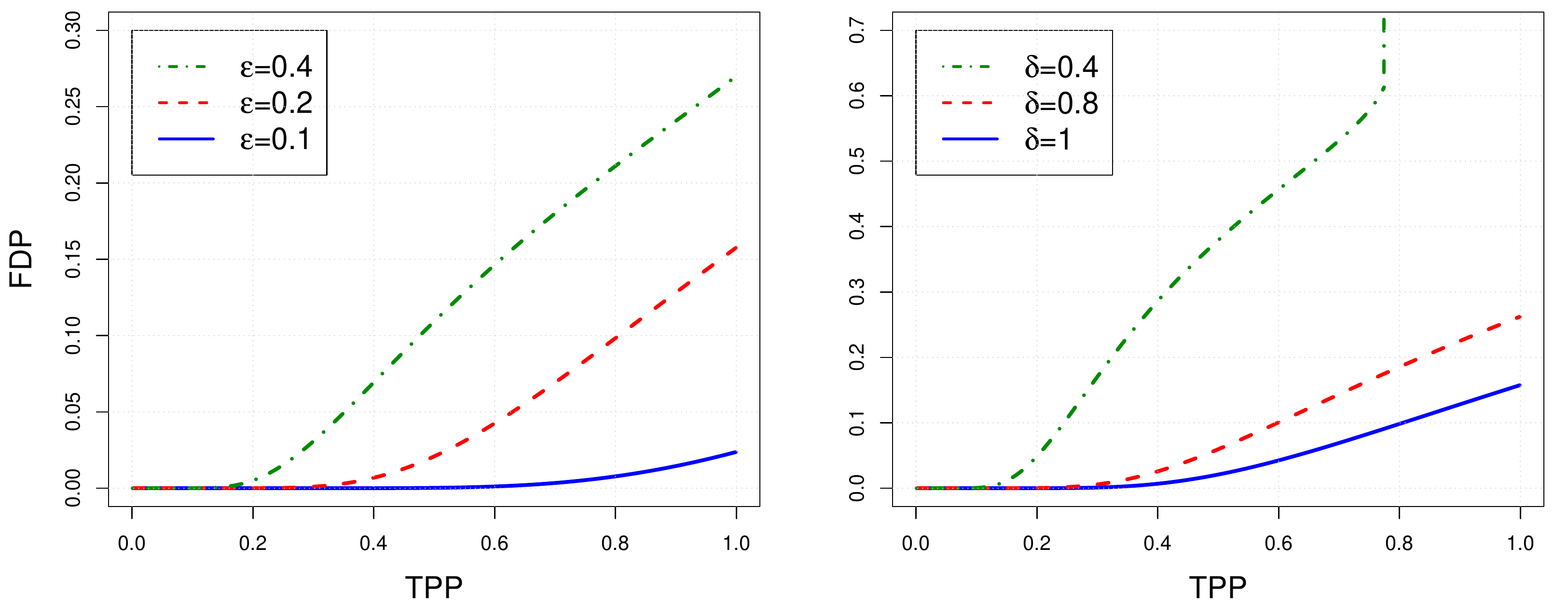}
\end{subfigure}
\begin{subfigure}[b]{\textwidth}
\centering
\includegraphics[scale=0.39]{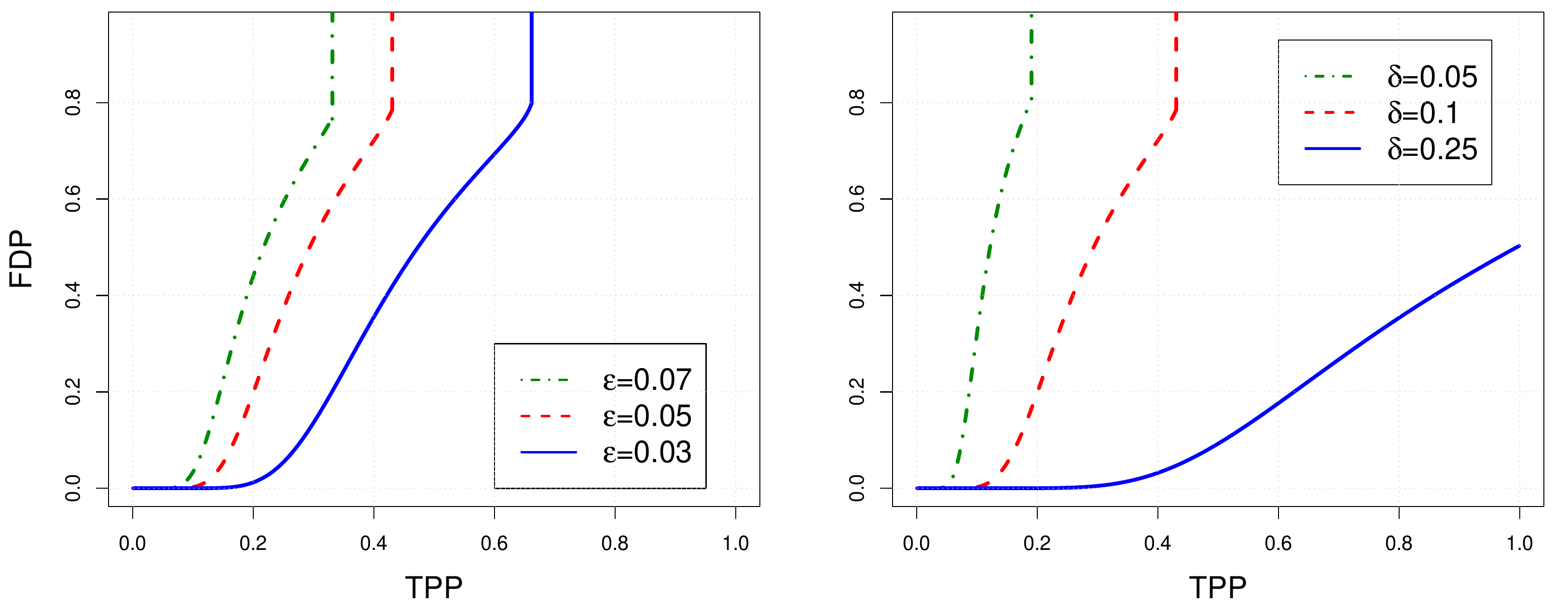}
\end{subfigure}
\caption{Top-left is with $\delta = 1$; top-right is with
  $\epsilon = 0.2$; bottom-left is with $\delta = 0.1$; and
  bottom-right is with $\epsilon = 0.05$. }
\label{fig:qfunction}
\end{figure}


\subsection{Numerical illustration}

Figure~\ref{fig:simu_FDP_TPP} provides the outcomes of numerical
simulations for finite values of $n$ and $p$ in the noiseless setup
where $\sigma=0$. For each of $n=p=1000$ and $n=p=5000$, we compute 10
independent Lasso paths and plot all pairs $(\power, \fdp)$ along the
way. In Figure~\ref{fig:simu_FDP_TPP}a we can see that when
$\power<0.8$, then the large majority of pairs $(\power, \fdp)$ along
these 10 paths are above the boundary. When $\power$ approaches one,
the average $\fdp$ becomes closer to the boundary and a fraction of
the paths fall below the line. As expected this proportion is
substantially smaller for the larger problem size.
\begin{figure}[!htp]
\centering
\begin{subfigure}[b]{0.49\textwidth}
\centering
\includegraphics[width=\textwidth]{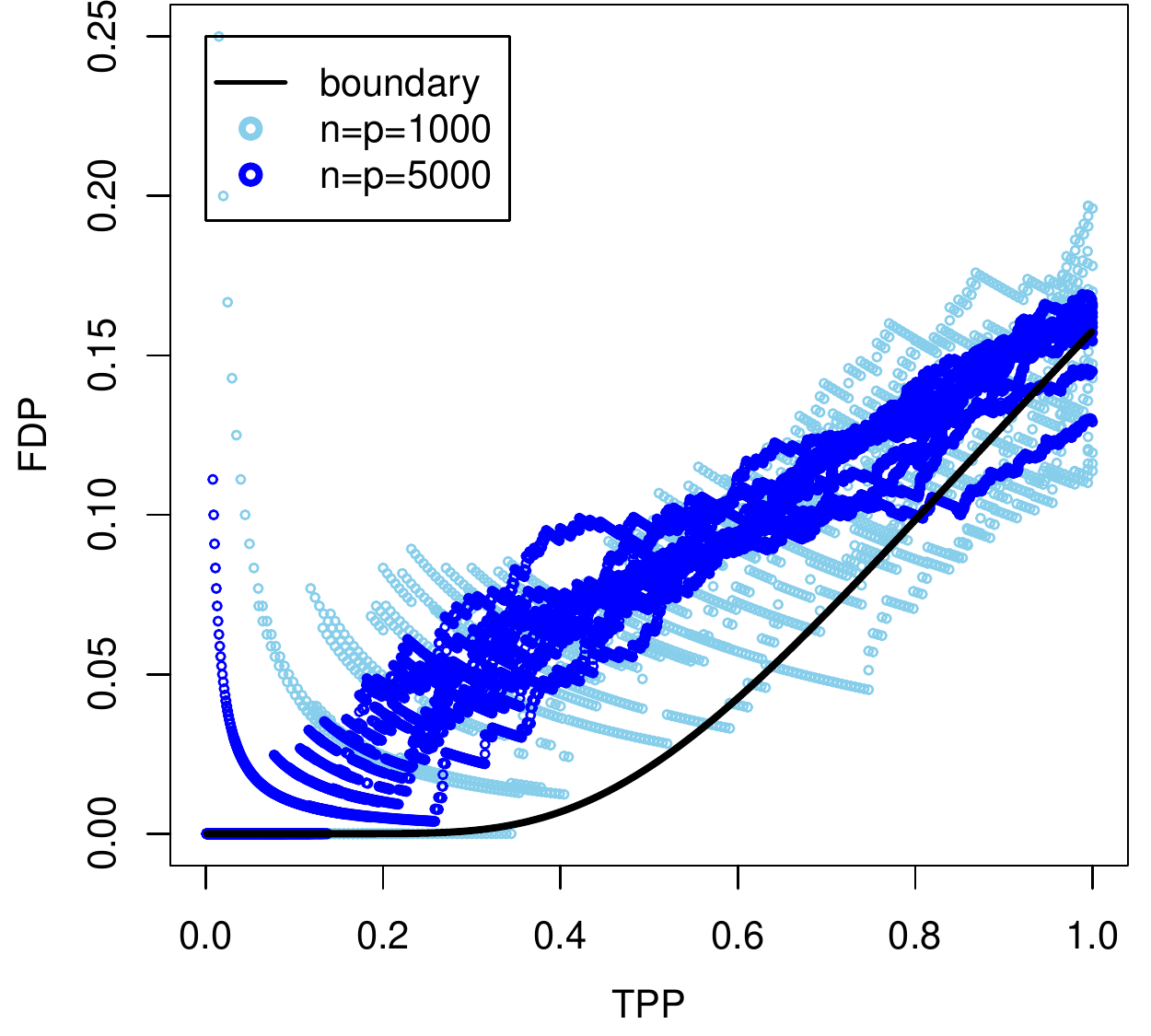}
\caption{Realized $(\power, \fdp)$ pairs}
\end{subfigure}
\hfill
\begin{subfigure}[b]{0.49\textwidth}
\centering
\includegraphics[width=\textwidth]{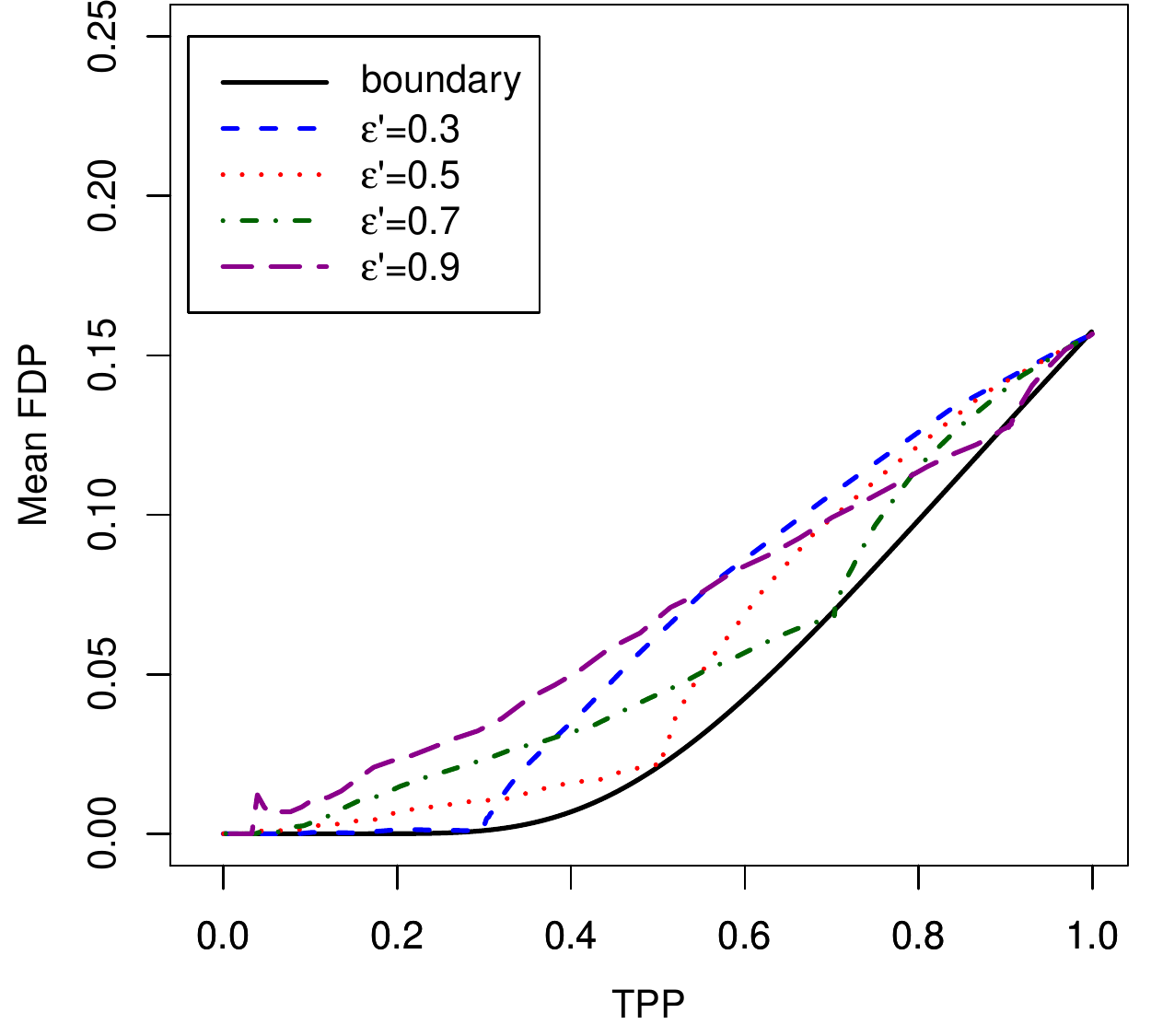}
\caption{Sharpness of the boundary}
\end{subfigure}
\caption {In both (a) and (b), $n/p = \delta = 1$, $\epsilon = 0.2$,
  and the noise level is $\sigma = 0$ (noiseless). (a) FDP vs.~TPP
  along 10 independent Lasso paths with
  $\P(\Pi=50) = 1 - \P(\Pi=0) = \epsilon$.  (b) Mean FDP vs.~mean TPP
  averaged at different values of $\lambda$ over 100 replicates for
  $n=p=1000$, $\P(\Pi = 0) = 1-\epsilon$ as before, and
  $\P(\Pi=50|\Pi \neq 0)= 1 - \P(\Pi=0.1 | \Pi \neq 0)= \epsilon'$.}
\label{fig:simu_FDP_TPP}
\end{figure}

\subsection{Sharpness}
\label{sec:sharpness}

The last conclusion from the theorem stems from the following
fact: take any point $(u, \mfdr(u))$ on the boundary curve; then we
can approach this point by fixing $\epsilon' \in (0, 1)$ and setting
the prior to be
\[
\Pi =
\begin{cases}
 M, \quad &\mbox{w.p.}~ \epsilon \cdot \epsilon',\\
 M^{-1}, &\mbox{w.p.}~ \epsilon \cdot (1-\epsilon'),\\
 0,             &\mbox{w.p.}~ 1 - \epsilon.
\end{cases}
\]
We think of $M$ as being very large so that the (nonzero) signals are
either very strong or very weak. In Appendix
\ref{sec:optimizing-tradeoff}, we prove that for any $u$ between 0 and 1 there is some fixed $\epsilon'=\epsilon'(u) > 0$ such that\footnote{In some cases $u$ should be bounded above by some constant strictly smaller than 1. See the previous footnote for details.}
\begin{equation}
\label{eq:lower}
\lim_{M \goto \infty} \,\, \lim_{n, p \goto \infty} \,\, (\power(\lambda), \fdp(\lambda)) \goto (u, \mfdr(u)),
\end{equation}
where convergence occurs in probability. This holds provided that
$\lambda \goto \infty$ in such a way that $M/\lambda \goto \infty$;
e.g.~$\lambda = \sqrt{M}$. Hence, the most favorable configuration is when
the signal is a mixture of {very}
strong and {very}
 weak
effect sizes because weak effects cannot be counted as false
positives, thus reducing the FDP.

Figure~\ref{fig:simu_FDP_TPP}b provides an illustration of
\eqref{eq:lower}. The setting is as in Figure~\ref{fig:simu_FDP_TPP}a
with $n = p = 1000$ and $\P(\Pi = 0) = 1-\epsilon$ except that, here,
conditionally on being nonzero the prior takes on the values $50$ and
$0.1$ with probability $\epsilon' \in \{0.3,0.5,0.7,0.9\}$ and
$1-\epsilon'$, respectively, so that we have a mixture of strong and
weak signals.  We observe that the true/false positive rate curve
nicely touches \textit{only} one point on the boundary depending on
the proportion $\epsilon'$ of strong signals .

\subsection{Technical novelties and comparisons with other works}
\label{sec:techn-contr}
The proof of Theorem~\ref{thm:min_fdr_given_power} is built on top of
the approximate message passing (AMP) theory developed in \cite{DMM09,
  ampdense,ampuniversality}, and requires nontrivial extensions. AMP
was originally designed as an algorithmic solution to compressive
sensing problems under random Gaussian designs. In recent years, AMP
has also found applications in robust statistics
\cite{donoho2013high,donoho2015variance}, structured principal
component analysis \cite{deshpande2014,montanari2014non}, and the
analysis of the stochastic block model \cite{deshpande2015}.  Having
said this, AMP theory is of crucial importance to us because it turns
out to be a very useful technique to rigorously study various
statistical properties of the Lasso solution whenever we employ a
\textit{fixed} value of the regularizing parameter $\lambda$
\cite{BM12,maleki2013asymptotic,mousavi2013}.

There are, however, major differences between our work and AMP
research. First and foremost, our paper is concerned with situations
where $\lambda$ is selected adaptively, i.e.~from the data; this is
clearly outside of the envelope of current AMP results. Second, we
are also concerned with situations where the noise variance can be
zero. Likewise, this is outside of current knowledge. These
differences are significant and as far as we know, our main result
cannot be seen as a straightforward extension of AMP theory. In
particular, we introduce a host of novel elements to deal, for
instance, with the \textit{irregularity} of the Lasso path. The
irregularity means that a variable can enter and leave the model
multiple times along the Lasso path \cite{efron2004least,
  tibshiraniLassopath} so that natural sequences of Lasso models are
not nested. This implies that a naive application of sandwiching
inequalities does not give the type of statements holding uniformly
over all $\lambda$'s that we are looking for.

Instead, we develop new tools to understand the ``continuity'' of the
support of $\widehat\bbeta(\lambda)$ as a function of $\lambda$. Since
the support can be characterized by the Karush-Kuhn-Tucker (KKT)
conditions, this requires establishing some sort of continuity of the
KKT conditions. Ultimately, we shall see that this comes down to
understanding the maximum distance---uniformly over $\lambda$ and
$\lambda'$---between Lasso estimates $\widehat\bbeta(\lambda)$ and
$\widehat\bbeta(\lambda')$ at close values of the regularizing
parameter. A complete statement of this important intermediate
  result is provided in Lemma \ref{lm:l2_local_continu} from Appendix
  \ref{sec:uniform-over-lambda}.

Our results can also be compared to the phase-transition curve from
\cite{DonTan05}, which was obtained under the same asymptotic regime
and describes conditions for perfect signal recovery in the noiseless
case. The solution algorithm there is the linear program, which
minimizes the $\ell_1$ norm of the fitted coefficients under equality
constraints, and corresponds to the Lasso solution in the limit of
$\lambda \goto 0$ (the end or bottom of the Lasso path).  The
conditions for perfect signal recovery by the Lasso turn out to be far
more restrictive than those related to this linear program. For
example, our FDP-TPP trade-off curves show that perfect recovery of an
infinitely large signal by Lasso is often practically impossible even
when $n\geq p$ (see Figure \ref{fig:qfunction}).  Interestingly, the
phase-transition curve also plays a role in describing the performance
of the Lasso, since it turns out that for signals dense enough not to
be recovered by the linear program, not only does the Lasso face the
problem of early false discoveries, it also hits a power limit for
arbitrary small values of $\lambda$ (see the discussion in
Appendix~\ref{sec:optimizing-tradeoff}).

Finally, we would like also to point out that some existing works have
investigated support recovery in regimes including linear sparsity
under random designs (see
e.g. \cite{wainwright2009information,reeves2013}). These interesting
results were, however, obtained by taking an information-theoretic
point of view and do not apply to computationally feasible methods
such as the Lasso.

\section{What's Wrong with Shrinkage?}
\label{sec:matter-with-ell_1}

\subsection{Performance of $\ell_0$ methods} 
\label{sec:l0}

We wrote earlier that not all methods share the same difficulties in
identifying those variables in the model.  If the signals are
sufficiently strong, some other methods, perhaps with exponential
computational cost, can achieve good model selection performance, see
e.g.~\cite{reeves2013}.  As an example, consider the simple
$\ell_0$-penalized maximum likelihood estimate,
\begin{equation}\label{eq:l0_ls}
  \widehat \bbeta_0 = \argmin_{\bb \in \R^p} \, \,\, \| \by - \bX \bb\|^2 + \lambda \,  \|\bb\|_0.
\end{equation}
Methods known as AIC, BIC and RIC (short for risk inflation criterion)
are all of this type and correspond to distinct values of the
regularizing parameter $\lambda$. It turns out that such fitting
strategies can achieve perfect separation in some cases.
\begin{theorem}\label{thm:l0}
  Under our working hypothesis, take $\epsilon < \delta$ for
  identifiability, and consider the two-point prior
\[
\Pi =
\begin{cases}
M, \quad &\mbox{w.p.}~ \epsilon,\\  
0,  \quad &\mbox{w.p.}~ 1-\epsilon.  
\end{cases}
\]
Then we can find $\lambda(M)$ such that in probability, the
discoveries of the $\ell_0$ estimator \eqref{eq:l0_ls} obey 
\[
\lim_{M \goto \infty} \, \lim_{n, p \goto \infty} \, \fdp = 0 \quad \text{and} 
\quad \lim_{M \goto \infty} \, \lim_{n, p \goto \infty} \, \power = 1.
\]
\end{theorem}
The proof of the theorem is in Appendix
\ref{sec:proof-theor-refthm:l}.  Similar conclusions will certainly
hold for many other non-convex methods, including SCAD and MC$+$ with
properly tuned parameters \cite{fan2001variable, zhang2010mc}.

\subsection{Some heuristic explanation}


In light of Theorem \ref{thm:l0}, we pause to discuss the cause
underlying the limitations of the Lasso for variable selection, which
comes from the pseudo-noise introduced by shrinkage.  As is
well-known, the Lasso applies some form of soft-thresholding. This
means that when the regularization parameter $\lambda$ is large, the
Lasso estimates are seriously biased downwards. Another way to put
this is that the residuals still contain much of the effects
associated with the selected variables. This can be thought of as
extra noise that we may want to call {\em shrinkage noise}. Now as
many strong variables get picked up, the shrinkage noise gets inflated
and its projection along the directions of some of the null variables
may actually dwarf the signals coming from the strong regression
coefficients; this is why null variables get picked up. Although our
exposition below dramatically lacks in rigor, it nevertheless
formalizes this point in some \textit{qualitative} fashion.  It is
important to note, however, that this phenomenon occurs in the linear
sparsity regime considered in this paper so that we have sufficiently
many variables for the shrinkage noise to build up and have a fold on
other variables that becomes competitive with the signal. In contrast,
under extreme sparsity and high SNR, both type I and II errors can be
controlled at low levels, see e.g.~\cite{ji2014rate}.

For simplicity, we fix the true support $\mathcal{T}$ to be a
deterministic subset of size $\epsilon \cdot p$, each nonzero
coefficient in $\mathcal T$ taking on a constant value $M > 0$.  Also,
assume $\delta > \epsilon$. Finally, since the noiseless case $\bm z =
\bm 0$ is conceptually perhaps the most difficult, suppose $\sigma =
0$. Consider the reduced Lasso problem first:
\[
\widehat \bbeta_{\mathcal{T}}(\lambda) = \argmin_{\bm b_{\mathcal{T}} \in \R^{\epsilon p}} \,\,\half \| \by - \bX_{\mathcal{T}} \bb_{\mathcal{T}} \|^2 + \lambda \,  \|\bb_{\mathcal{T}}\|_1.
\]
This (reduced) solution $\widehat \bbeta_{\mathcal{T}}(\lambda)$ is
independent from the other columns $\bX_{\overline{\mathcal{T}}}$
(here and below $\overline{\mathcal{T}}$ is the complement of
$\mathcal{T}$). Now take $\lambda$ to be of the same magnitude as $M$ so that roughly half of the signal variables are selected.
The KKT conditions state that
\[
-\lambda \bm{1} \le  \bX_{\mathcal{T}}^\T (\by - \bX_{\mathcal{T}} \widehat\bbeta_{\mathcal{T}})  \le \lambda \bm{1},
\]
where $\bm{1}$ is the vectors of all ones.  Note that if
$|\bX_j^\T (\by - \bX_{\mathcal{T}} \widehat\bbeta_{\mathcal{T}})| \le
\lambda$
for all $j \in \overline{\mathcal T}$, then extending
$\widehat \bbeta_{\mathcal{T}}(\lambda)$ with zeros would be the
solution to the full Lasso problem---with all variables included as
potential predictors---since it would obey the KKT conditions for the
full problem. A first simple fact is this: for
  $j \in \overline{\mathcal T}$, if
\begin{equation}\label{eq:break_kkt}
|\bX_j^\T (\by - \bX_{\mathcal{T}} \widehat\bbeta_{\mathcal{T}})| >\lambda,
\end{equation}
then $\bX_j$ must be selected by the incremental Lasso with design
variables indexed by $\mathcal T \cup \{j\}$. Now we make an
assumption which is heuristically reasonable: any $j$ obeying
\eqref{eq:break_kkt} has a reasonable chance to be selected in the
full Lasso problem with the same $\lambda$ (by this, we mean with some
probability bounded away from zero). We argue in favor of this
heuristic later. 

Following our heuristic, we would need to argue that
\eqref{eq:break_kkt} holds for a number of variables in
$\overline{\mathcal{T}}$ {\em linear} in $p$.  Write
\[
\bX_{\mathcal{T}}^\T (\by - \bX_{\mathcal{T}} \widehat\bbeta_{\mathcal{T}}) = \bX_{\mathcal{T}}^\T (\bX_{\mathcal{T}} \bbeta_{\mathcal{T}} - \bX_{\mathcal{T}} \widehat\bbeta_{\mathcal{T}}) = \lambda \bg_{\mathcal{T}},
\]
where $\bg_{\mathcal{T}}$ is a subgradient of the $\ell_1$ norm at
$\widehat\bbeta_{\mathcal{T}}$. Hence, $\bbeta_{\mathcal{T}} - \widehat\bbeta_{\mathcal{T}} =  \lambda (\bX_{\mathcal{T}}^\T \bX_{\mathcal{T}})^{-1} \bg_{\mathcal{T}}$ and 
\[
\bX_{\mathcal{T}}(\bbeta_{\mathcal{T}} - \widehat\bbeta_{\mathcal{T}}) = \lambda \bX_{\mathcal{T}}(\bX_{\mathcal{T}}^\T \bX_{\mathcal{T}})^{-1} \bg_{\mathcal{T}}.
\]
Since $\delta > \epsilon$, $\bX_{\mathcal{T}}(\bX_{\mathcal{T}}^\T
\bX_{\mathcal{T}})^{-1}$ has a smallest singular value bounded away
from zero (since $\bX_{\mathcal{T}}$ is a fixed random matrix with
more rows than columns).  Now because we make about half discoveries,
the subgradient takes on the value one (in magnitude) at about
$\epsilon \cdot p/2$ times. Hence, with high probability,
\[
\|\bX_{\mathcal{T}} (\bbeta_{\mathcal{T}} - \widehat\bbeta_{\mathcal{T}})\| \ge \lambda \cdot c_0 \cdot \|\bg_{\mathcal{T}}\| \ge \lambda \cdot c_1 \cdot {p}
\]
for some constants $c_0, c_1$ depending on $\epsilon$ and $\delta$. 

Now we use the fact that $\widehat \bbeta_{\mathcal{T}}(\lambda)$ is
independent of $\bX_{\overline{\mathcal{T}}}$.  For any
$j \notin \mathcal{T}$, it follows that
\[
\bX_j^\T (\by - \bX_{\mathcal{T}} \widehat\bbeta_{\mathcal{T}}) = \bX_j^\T \bX_{\mathcal{T}} (\bbeta_{\mathcal{T}} - \widehat\bbeta_{\mathcal{T}})
\]
is conditionally normally distributed with mean zero and variance
\[
\frac{\|\bX_{\mathcal{T}} (\bbeta_{\mathcal{T}} -
  \widehat\bbeta_{\mathcal{T}})\|^2}{n} \ge \frac{ c_1 \lambda^2 p}{n}
= c_2 \cdot \lambda^2.
\]
In conclusion, the probability that
$\bX_j^\T (\by - \bX_{\mathcal{T}} \widehat\bbeta_{\mathcal{T}})$ has
absolute value larger than $\lambda$ is bounded away from 0. Since
there are $(1-\epsilon)p$ such $j$'s, their expected number is linear
in $p$.  This implies that by the time half of the true variables are
selected, we already have a non-vanishing FDP.  Note that when
$|\mathcal{T}|$ is not linear in $p$ but smaller,
e.g.~$|\mathcal{T}| \le c_0 n/\log p$ for some sufficiently small
constant $c_0$, the variance is much smaller because the estimation
error
$\|\bX_{\mathcal{T}} (\bbeta_{\mathcal{T}} -
\widehat\bbeta_{\mathcal{T}})\|^2$
is much lower, and this phenomenon does not occur.

Returning to our heuristic, we make things simpler by considering 
alternatives: (a) if very few extra variables in
$\overline{\mathcal T}$ were selected by the full Lasso, then the
value of the prediction $\bX \widehat\bbeta$ would presumably be close
to that obtained from the reduced model.
In other words, the residuals $\by - \bX \widehat\bbeta$ from the full
problem should not differ much from those in the reduced
problem. Hence, for any $j$ obeying \eqref{eq:break_kkt}, $\bX_j$
would have a high correlation with $\by - \bX \widehat\bbeta$. Thus
this correlation has a good chance to be close to $\lambda$, or
actually be equal to $\lambda$. Equivalently, $\bX_j$ would likely be
selected by the full Lasso problem. (b) If on the other hand, the
number of variables selected from $\overline{\mathcal T}$ by the full
Lasso were a sizeable proportion of $|\mathcal{T}|$, we would have
lots of false discoveries, which is our claim.

In a more rigorous way, AMP claims that under
our working hypothesis, the Lasso estimates $\widehat\beta_j(\lambda)$
are, in a certain sense, asymptotically distributed as
$\eta_{\alpha\tau}(\beta_j + \tau W_j)$ for most $j$ and $W_j$'s
independently drawn from $\mathcal{N}(0,1)$. The positive constants
$\alpha$ and $\tau$ are uniquely determined by a pair of nonlinear
equations parameterized by $\epsilon, \delta, \Pi, \sigma^2$, and
$\lambda$. Suppose as before that all the nonzero coefficients of
$\bbeta$ are large in magnitude, say they are all equal to $M$. When
about half of them appear on the path, we have that $\lambda$ is just
about equal to $M$. A consequence of the AMP equations is that $\tau$
is also of this order of magnitude. Hence, under the null we have that
$(\beta_j + \tau W_j)/M \sim \mathcal{N}(0, (\tau/M)^2)$ while under
the alternative, it is distributed as $\mathcal{N}(1, (\tau/M)^2)$.
Because, $\tau/M$ is bounded away from zero, we see that false
discoveries are bound to happen.

Variants of the Lasso and other $\ell_1$-penalized methods, including
$\ell_1$-penalized logistic regression and the Dantzig selector, also
suffer from this ``shrinkage to noise'' issue. 

\section{Discussion}
\label{sec:discussion}

We have evidenced a clear trade-off between false and true positive
rates under the assumption that the design matrix has i.i.d.~Gaussian
entries. It is likely that there would be extensions of this result to
designs with general \iid~sub-Gaussian entries as strong evidence
suggests that the AMP theory may be valid for such larger classes,
see \cite{ampuniversality}.  It might also be of interest to study the
Lasso trade-off diagram under correlated random designs.

  As we previously mentioned in the introduction, a copious
  body of literature considers the Lasso support recovery under
  Gaussian random designs, where the sparsity of the signal is often
  assumed to be \textit{sub-linear} in the ambient dimension
  $p$. Recall that if all the nonzero signal components have
  magnitudes at least $c\sigma\sqrt{2\log p}$ for some unspecified
  numerical constant $c$ (which would have to exceed one), the results
  from \cite{lassogaussian} conclude that, asymptotically, a sample
  size of $n \ge (2 + o(1)) k \log p$ is both necessary and sufficient
  for the Lasso to obtain perfect support recovery.  What does these
  results say for finite values of $n$ and $p$?
  Figure~\ref{fig:compare_dis} demonstrates the performance of the
  Lasso under a moderately large $250 \times 1000$
random Gaussian design.
Here, we consider a very sparse signal, where only $k=18$ regression
coefficients are nonzero,
$\beta_1 = \cdots = \beta_{18} = 2.5\sqrt{2 \log p} \approx 9.3$,
$\beta_{19} = \cdots = \beta_{1000} = 0$, and the noise variance is
$\sigma^2 = 1$.  Since $k = 18$ is smaller than $ {n}/{2\log p}$ and
$\beta$ is substantially larger than $\sqrt{2 \log p}$ one might
expect that Lasso would recover the signal support. However, Figure
\ref{fig:compare} (left) 
shows that this might not be the case. We see that the Lasso includes
five false discoveries before all true predictors are included, which
leads to an FDP of 21.7\% by the time the power (TPP) reaches 1.
Figure~\ref{fig:compare_dis} (right) summarizes the outcomes from 500
independent experiments, and shows that the average FDP reaches 13.4\%
when $\text{TPP}=1$.  With these dimensions, perfect recovery is not
guaranteed even in the case of `infinitely' large signals (no
noise). In this case, perfect recovery occurs in only 75\% of all
replicates and the averaged FDP at the point of full power is equal to
1.7\%, which almost perfectly agrees with the boundary FDP provided in
Theorem \ref{thm:min_fdr_given_power}. Thus, quite surprisingly, our
results obtained under a {\it linear sparsity regime} apply to sparser
regimes, and might prove useful across a wide range of sparsity
levels. 

\begin{figure}[htp!]
\centering
\includegraphics[scale=0.55]{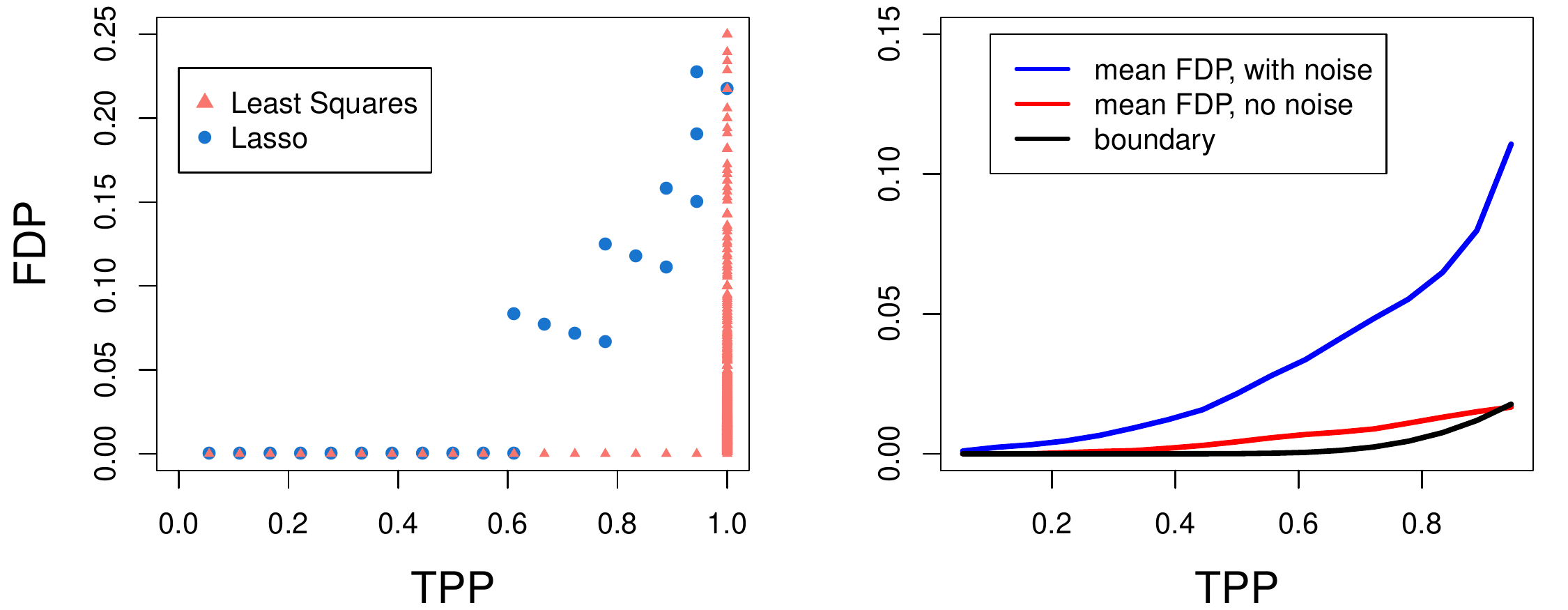}
\caption{Simulation setup: $n=250$, $p=1000$,
  $\beta_1=\cdots=\beta_{18}=2.5 \sqrt{2 \log p} \approx 9.3$ (the
  other coefficients all vanish), $\sigma^2=1$ (with noise) and
  $\sigma^2=0$ (no noise). Left: noisy case. True positive and false
  positive rates along a single realization of the Lasso path. The
  least squares path is obtained by ordering least squares estimates
  from a model including the first 50 variables selected by the
  Lasso. Right: mean FDP as a function of TPP. The mean FDP was obtained
  by averaging over 500 independent trials.}
\label{fig:compare_dis}
\end{figure}




Of concern in this paper are statistical properties regarding the
number of true and false discoveries along the Lasso path but it would
also be interesting to study perhaps finer questions such as this:
when does the first noise variable get selected? Consider
Figure~\ref{fig:T_high_snr}: there, $n = p = 1000$, $\sigma^2=1$,
$\beta_1=\cdots=\beta_k=50$ (very large SNR) and $k$ varies from 5 to
150. In the very low sparsity regime, all the signal variables are
selected before any noise variable. When the number $k$ of signals
increases we observe early false discoveries, which may occur for
values of $k$ smaller than $n/(2\log p)$. However, the average rank of
the first false discovery is substantially smaller than $k$ only after
$k$ exceeds $n/(2\log p)$. Then it keeps on decreasing as $k$
continues to increase, a phenomenon not explained by any result we are
aware of.
In the linear sparsity regime, it would be interesting to derive a
prediction for the average time of the first false entry, at least in
the noiseless case.
\begin{figure}[h!]
\centering
\includegraphics[scale=0.75]{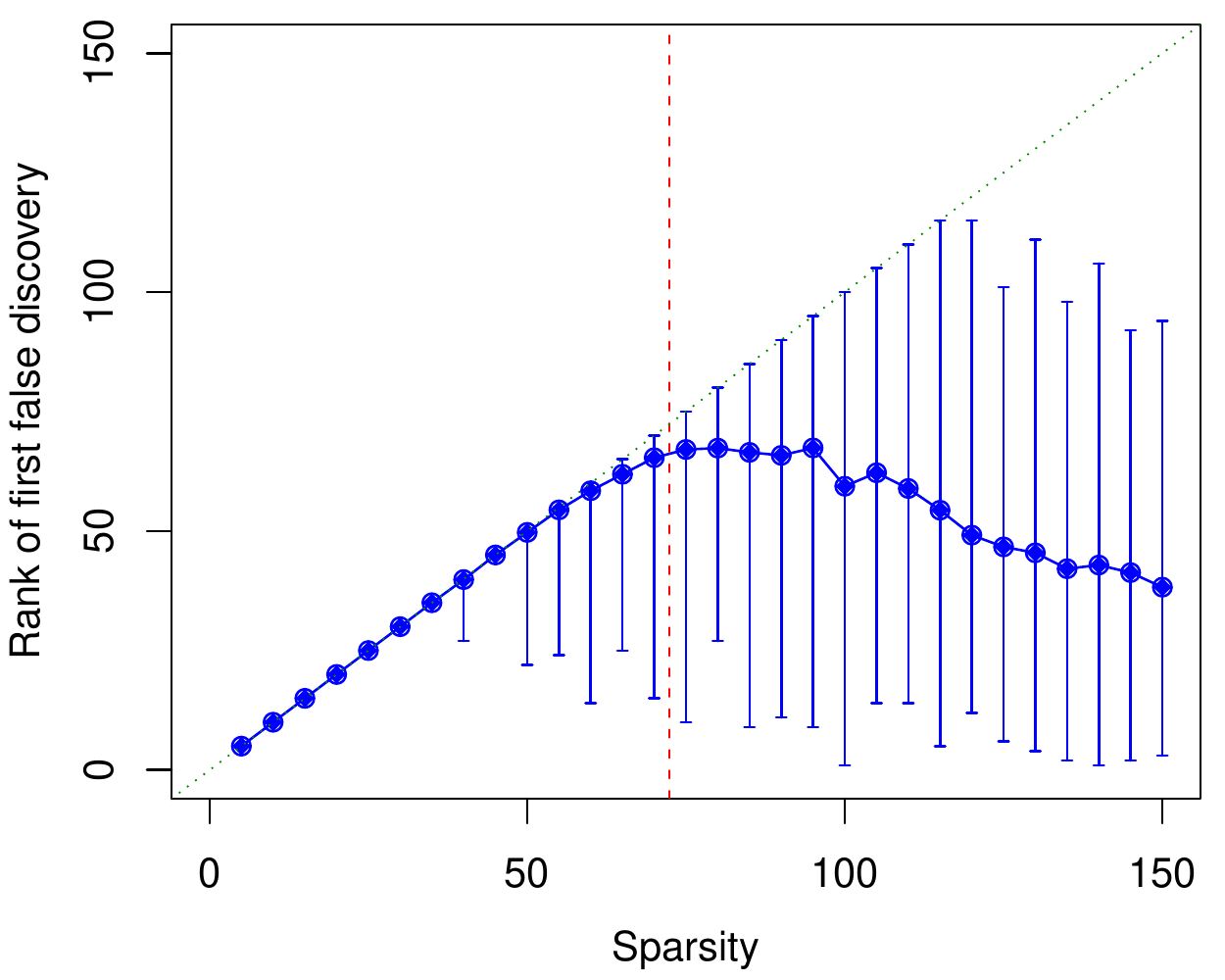}
\caption{Rank of the first false discovery.  Here, $n = p = 1000$ and
  $\beta_1=\cdots=\beta_k=50$ for $k$ ranging from 5 to 150
  ($\beta_i = 0$ for $i > k$). We plot averages from 100 independent
  replicates and display the range between minimal and maximal
  realized values. The vertical line is placed at $k = n/(2\log p)$
and the $45^\circ$ line passing through the origin is shown for
convenience.}
\label{fig:T_high_snr}
\end{figure}

Methods that are computationally efficient and also enjoy good
  model performance in the linear sparsity regime would be highly
  desirable. (The Lasso and the $\ell_0$ method each enjoys one
  property but not the other.)  While it is beyond our scope to
  address this problem, we conclude the discussion by considering
  marginal regression, a technique widely used in practice and not
  computationally demanding. A simple analysis shows that marginal
  regression suffers from the same issue as the Lasso under our
  working hypothesis. To show this, examine the noiseless case
  ($\sigma = 0$) and assume $\beta_1 = \cdots = \beta_k = M$ for some
  constant $M > 0$. It is easy to see that the marginal statistic $\bm X_j^\T \bm y$ for the $j$th variable is
  asymptotically distributed as $\mathcal{N}(0, \tilde\sigma^2)$, where $\tilde\sigma = M \sqrt{(k-1)/n}$, if $j$ is a true null and
  $\mathcal{N}(M, \tilde\sigma^2)$ otherwise. In the linear sparsity regime,
  where $k/n$ tends to a constant, the mean shift $M$ and standard
  deviation $\tilde\sigma$ have comparable magnitudes.  As a result, nulls and non-nulls are also interspersed on the marginal regression path, so that we would have either high FDR or low power.

\section*{Acknowledgements}
W.~S.~would like to thank Andrea Montanari for helpful discussions. We thank Yuxin Chen for helpful comments about an early
version of the manuscript. Most of this research was performed while
W.~S.~was at Stanford University.

\bibliographystyle{abbrv}
\bibliography{ref}

\appendix
 
  \section{Road Map to the Proofs}
\label{sec:proof-roadmap}

In this section, we provide an overview of the proof of
  Theorem \ref{thm:min_fdr_given_power}, presenting all the key steps
  and ingredients. Detailed proofs are distributed in Appendices
  \ref{sec:uniform-over-lambda}--\ref{sec:proof-theor-refthm:m}.  At a
  high level, the proof structure has the following three elements:
\begin{enumerate}
\item Characterize the Lasso solution at a fixed $\lambda$
  asymptotically, predicting the (non-random) asymptotic values of the
  FDP and of the TPP denoted by $\ffdp(\lambda)$ and
  $\fpower(\lambda)$, respectively. These limits depend depend on
  $\Pi, \delta, \epsilon$ and $\sigma$.
\item Exhibit uniform convergence over $\lambda$ in the sense that
\[
\sup_{\lambda_{\min} \le \lambda \le \lambda_{\max}} 
\left| \fdp(\lambda) - \ffdp(\lambda) \right| \plim 0, 
\]
and similarly for $\power(\lambda)$. A consequence is that in the
limit, the asymptotic trade-off between true and false positive rates
is given by the $\lambda$-parameterized curve
$(\fpower(\lambda), \ffdp(\lambda))$.

\item The trade-off curve from the last step depends on the prior
  $\Pi$. The last step optimizes it by varying $\lambda$ and $\Pi$.
\end{enumerate}
Whereas the last two steps are new and present some technical
challenges, the first step is accomplished largely by resorting to
off-the-shelf AMP theory.  We now present each step in turn.
Throughout this section we work under our working hypothesis and take
the noise level $\sigma$ to be positive.

\paragraph{Step 1} 
First, Lemma \ref{lm:amp_theory} below accurately predicts the
asymptotic limits of FDP and TPP at a fixed $\lambda$ . This lemma is
borrowed from \cite{supp}, which follows from Theorem 1.5 in
\cite{BM12} in a natural way, albeit with some effort spent in
resolving a continuity issue near the origin. Recall that
$\eta_t(\cdot)$ is the soft-thresholding operator defined as
$\eta_{t}(x) = \sgn(x)(|x| - t)_+$, and $\Pi^\star$ is the
distribution of $\Pi$ conditionally on being nonzero;
\[
\Pi = \begin{cases}
\Pi^\star, \quad & \mbox{w.p.}~ \epsilon,\\
0, & \mbox{w.p.}~ 1-\epsilon.
\end{cases}
\]
Denote by $\alpha_0$ the unique root of $(1+t^2)\Phi(-t) -t\phi(t) = \delta/2$.
\begin{lemma}[Theorem 1 in \cite{supp}; see also Theorem 1.5 in \cite{BM12}]\label{lm:amp_theory}
The Lasso solution with a fixed $\lambda > 0$ obeys
\[
\frac{V(\lambda)}{p} \plim 2(1 - \epsilon)\Phi(-\alpha), \qquad
\frac{T(\lambda)}{p} \plim \epsilon \cdot \P(|\Theta^\star + \tau W| > \alpha\tau),
\]
where $W$ is $\mathcal{N}(0,1)$ independent of $\Theta$, and $\tau > 0, \alpha > \max\{ \alpha_0, 0\}$ is the unique solution to
\begin{equation}\label{eq:amp_eqn}
\begin{aligned}
\tau^2 &= \sigma^2 + \frac{1}{\delta}\mathbb{E} \left( \eta_{\alpha\tau}(\Theta + \tau W) - \Theta \right)^2\\
\lambda &= \left(1 - \frac{1}{\delta}\P (|\Theta + \tau W| > \alpha\tau) \right) \alpha\tau.\\
\end{aligned}
\end{equation}
Note that both $\tau$ and $\alpha$ depend on $\lambda$.
\end{lemma}
We pause to briefly discuss how Lemma~\ref{lm:amp_theory} follows from
Theorem 1.5 in \cite{BM12}. There, it is rigorously proven that the
joint distribution of $(\bbeta, \widehat\bbeta)$ is, in some sense,
asymptotically the same as that of
$(\bbeta, \eta_{\alpha\tau}(\bbeta + \tau \bm W))$, where $\bm W$ is a
$p$-dimensional vector of \iid standard normals independent of
$\bbeta$, and where the soft-thresholding operation acts in a
componentwise fashion. Roughly speaking, the Lasso estimate
$\widehat\beta_j$ looks like $\eta_{\alpha\tau}(\beta_j + \tau W_j)$,
so that we are applying soft thresholding at level $\alpha\tau$ rather
than $\lambda$ and the noise level is $\tau$ rather than $\sigma$.
With these results in place, we informally obtain 
\begin{align*}
V(\lambda)/p = \#\{j: \widehat\beta_j \ne 0, \beta_j = 0\}/p 
& \approx \P(\eta_{\alpha\tau}(\Pi + \tau W) \ne 0, \Pi = 0)\\
& = (1-\epsilon) \, \P(|\tau W| > \alpha\tau)\\
& = 2(1-\epsilon)\, \Phi(-\alpha).
\end{align*}
Similarly,
$T(\lambda)/p \approx \epsilon \, \P(|\Pi^\star + \tau W| > \alpha
\tau)$. For details, we refer the reader to Theorem 1 in \cite{supp}.

\paragraph{Step 2} 
Our interest is to extend this convergence result uniformly over a
range of $\lambda$'s. The proof of this step is the subject of
Section~\ref{sec:uniform-over-lambda}.
\begin{lemma}\label{lm:lamb_uniform}
For any fixed $0 < \lambda_{\min} < \lambda_{\max}$, the convergence
of $V(\lambda)/p$ and $T(\lambda)/p$ in Lemma~\ref{lm:amp_theory} is
uniform over $\lambda \in [\lambda_{\min}, \lambda_{\max}]$.  
\end{lemma}
Hence, setting
\begin{equation}
\label{eq:def}
\fd(\lambda) := 2(1-\epsilon)\Phi(-\alpha), \quad \td(\lambda) := \epsilon\P(|\Pi^\star+
\tau W| > \alpha\tau)
\end{equation}
we have 
\[
\sup_{\lambda_{\min} \le \lambda \le \lambda_{\max}} \left| \frac{V(\lambda)}{p} - \fd(\lambda) \right| \plim 0,
\]
and 
\[
\sup_{\lambda_{\min} \le \lambda \le \lambda_{\max}} \left| \frac{T(\lambda)}{p} - \td(\lambda) \right| \plim 0.
\]
To exhibit the trade-off between FDP and TPP, we can therefore focus
on the far more amenable quantities $\fd(\lambda)$ and $\td(\lambda)$
instead of $V(\lambda)$ and $T(\lambda)$.  Since
$\fdp(\lambda) = V(\lambda)/(V(\lambda) + T(\lambda))$ and
$\power(\lambda) = T(\lambda)/|\{j : \beta_j \neq 0\}|$, this gives
\[
\sup_{\lambda_{\min} \le \lambda \le \lambda_{\max}} \left|
  \fdp(\lambda) - \ffdp(\lambda) \right| \plim 0, \quad
\ffdp(\lambda)= \frac{\fd(\lambda)}{\fd(\lambda) + \td(\lambda)},
\]
and
\[
\sup_{\lambda_{\min} \le \lambda \le \lambda_{\max}} \left|
  \power(\lambda) - \fpower(\lambda) \right| \plim 0, \quad
\fpower(\lambda)= \frac{\td(\lambda)}{\epsilon}, 
\]
so that $\ffdp(\lambda)$ and $\fpower(\lambda)$ are the predicted FDP
and TPP. (We shall often hide the dependence on $\lambda$.)

\paragraph{Step 3} As remarked earlier, both $\fpower(\lambda)$ and
$\ffdp(\lambda)$ depend on $\Pi, \delta, \epsilon$ and $\sigma$. In
Appendix \ref{sec:optimizing-tradeoff}, we will see that we can
parameterize the trade-off curve $(\fpower(\lambda), \ffdp(\lambda))$
by the true positive rate so that there is a function $q^{\Pi}$
obeying $q^{\Pi}(\fpower) = \ffdp$; furthermore, this function depends
on $\Pi$ and $\sigma$ only through $\Pi/\sigma$. Therefore,
realizations of the FDP-TPP pair fall asymptotically arbitrarily close
to $q^{\Pi}$.  It remains to optimize the curve $q^{\Pi}$ over
$\Pi/\sigma$. Specifically, the last step in Appendix
\ref{sec:optimizing-tradeoff} characterizes the envelope $q^\star$
formally given as
\[
q^\star(u; \delta, \epsilon) = \inf q^{\Pi}(u; \delta, \epsilon),
\]
where the infimum is taken over all feasible priors $\Pi$. The key
ingredient in optimizing the trade-off is given by Lemma
\ref{lm:convexity}.

Taken together, these three steps sketch the basic strategy for
proving Theorem \ref{thm:min_fdr_given_power}, and the remainder of
the proof is finally carried out in Appendix
\ref{sec:proof-theor-refthm:m}. In particular, we also establish the
noiseless result ($\sigma=0$) by using a sequence of approximating
problems with noise levels approaching zero.

\section{For All Values of $\lambda$ Simultaneousy}
\label{sec:uniform-over-lambda}
In this section we aim to prove Lemma \ref{lm:lamb_uniform} and, for the moment, take $\sigma > 0$.  Also, we shall frequently use results in \cite{BM12}, notably, Theorem 1.5, Lemma 3.1, Lemma 3.2, and Proposition 3.6 therein. Having said this, most of our proofs are rather self-contained, and the strategies accessible to readers who have not yet read \cite{BM12}. We start by stating two auxiliary lemmas below, whose proofs are deferred to Section~\ref{sec:proof-auxil-lemm}.

\begin{lemma}\label{lm:weak_corre}
For any $c > 0$, there exists a constant $r_c > 0$ such that for any arbitrary $r > r_c$,
\[
\sup_{\|\bm u\| = 1}\#\left\{1 \le j \le p: |\bX_j^{\transp} \bm u| > \frac{r}{\sqrt{n}} \right\} \le c p 
\]
holds with probability tending to one. 
\end{lemma}

A key ingredient in the proof of Lemma \ref{lm:lamb_uniform} is, in a certain sense, the uniform continuity of the support of $\widehat\bbeta(\lambda)$. This step is justified by the auxiliary lemma below which demonstrates that the Lasso estimates are uniformly continuous in $\ell_2$ norm.

\begin{lemma}\label{lm:l2_local_continu}
  Fixe $0 < \lambda_{\min} < \lambda_{\max}$. Then there is a constant
  $c$ such for any $\lambda^- < \lambda^+$ in
  $\left[ \lambda_{\min}, \lambda_{\max} \right]$,
\begin{equation}\nonumber
\sup_{\lambda^- \le \lambda \le \lambda^+} \left\| \widehat{\bm\beta}(\lambda) - \widehat{\bm\beta}(\lambda^-) \right\| \le c \sqrt{(\lambda^+-\lambda^-)p}
\end{equation}
holds with probability tending to one.
\end{lemma}

\begin{proof}[Proof of Lemma~\ref{lm:lamb_uniform}]
  We prove the uniform convergence of $V(\lambda)/p$ and similar
  arguments apply to $T(\lambda)/p$. To begin with, let
  $\lambda_{\min} = \lambda_0 < \lambda_1 < \cdots < \lambda_m =
  \lambda_{\max}$
  be equally spaced points and set
  $\Delta := \lambda_{i+1} - \lambda_i = (\lambda_{\max} -
  \lambda_{\min})/m$;
  the number of knots $m$ shall be specified later. It follows from 
  Lemma~\ref{lm:amp_theory} that 
\begin{equation}\label{eq:m_union}
\max_{0 \le i \le m} |V(\lambda_i)/p - \fd(\lambda_i)| \plim 0
\end{equation}
by a union bound.  Now, according to Corollary 1.7 from \cite{BM12},
the solution $\alpha$ to equation \eqref{eq:amp_eqn} is continuous in
$\lambda$ and, therefore, $\fd(\lambda)$ is also continuous on
$[\lambda_{\min}, \lambda_{\max}]$ . Thus, for any constant
$\omega > 0$, the equation
\begin{equation}\label{eq:rho_contin}
|\fd(\lambda) - \fd(\lambda')| \le \omega
\end{equation}
holds for all
$\lambda_{\min} \le \lambda, \lambda' \le \lambda_{\max}$ satisfying
$|\lambda - \lambda'| \le 1/m$ provided  $m$ is sufficiently
large. We now aim to show that if $m$ is
sufficiently large (but fixed), then
\begin{equation}\label{eq:support_const}
\max_{0 \le i < m} \sup_{\lambda_i \le \lambda \le \lambda_{i+1}} |V(\lambda)/p - V(\lambda_{i})/p| \le \omega
\end{equation}
holds with probability approaching one as $p \goto \infty$.  Since
$\omega$ is arbitrary small, combining \eqref{eq:m_union},
\eqref{eq:rho_contin}, and \eqref{eq:support_const} gives uniform
convergence by applying the triangle inequality.

Let $\mathcal{S}(\lambda)$ be a short-hand for
$\supp{\widehat\bbeta(\lambda)}$. Fix $0 \le i < m$ and put
$\lambda^- = \lambda_i$ and $\lambda^+ = \lambda_{i+1}$. For any
$\lambda \in [\lambda^-, \lambda^+]$, 
\begin{equation}\label{eq:delta_set}
|V(\lambda) - V(\lambda^-)| \le |\mathcal{S}(\lambda)\setminus \mathcal{S}(\lambda^-)| + |\mathcal{S}(\lambda^-)\setminus \mathcal{S}(\lambda)|.
\end{equation}
Hence, it suffices to give upper bound about the sizes of
$\mathcal{S}(\lambda)\setminus \mathcal{S}(\lambda^-)$ and
$\mathcal{S}(\lambda^-)\setminus \mathcal{S}(\lambda)$. We start with $|\mathcal{S}(\lambda)\setminus \mathcal{S}(\lambda^-)|$. 

The KKT optimality conditions for the Lasso solution state that there
exists a subgradient
$\bg(\lambda) \in \partial \|\widehat\bbeta(\lambda)\|_1$ obeying
\[
\bX^{\transp}(\by - \bX\widehat\bbeta(\lambda)) = \lambda \bg(\lambda)
\]
for each $\lambda$. Note that $g_j(\lambda) = \pm 1$ if $j \in \mathcal{S}(\lambda)$. As mentioned earlier, our strategy is to establish some sort of continuity of the KKT conditions with respect to $\lambda$. To this end, let
\[
\bm u = \frac{\bm X \left( \widehat\bbeta(\lambda) - \widehat\bbeta(\lambda^-) \right)}{\left\| \bX \left( \widehat\bbeta(\lambda) - \widehat\bbeta(\lambda^-) \right) \right\|}
\]
be a point in $\R^n$ with unit $\ell_2$ norm. Then for each $j \in \mathcal{S}(\lambda)\setminus \mathcal{S}(\lambda^-)$, we have
\[
|\bX_j^{\transp} \bm u |= \frac{\left| \bX_j^{\transp} \bX(\widehat\bbeta(\lambda) - \widehat\bbeta(\lambda^-)) \right|}{\|\bX (\widehat\bbeta(\lambda) - \widehat\bbeta(\lambda^-))\|} = \frac{|\lambda g_j(\lambda) - \lambda^- g_j(\lambda^-)|}{\| \bX (\widehat\bbeta(\lambda) - \widehat\bbeta(\lambda^-))\|} \ge \frac{\lambda - \lambda^- |g_j(\lambda^-)|}{\| \bX (\widehat\bbeta(\lambda) - \widehat\bbeta(\lambda^-))\|}.
\]
Now, given an arbitrary constant $a > 0$ to be determined later,
either $|g_j(\lambda^-)| \in [1-a, 1)$ or
$|g_j(\lambda^-)| \in [0, 1-a)$. In the first case ((a) below) note
that we exclude $|g_j(\lambda^-)| = 1$ because for random designs,
when $j \notin \mathcal{S}(\lambda)$ the equality
$|\bX_j^\T (\by - \bX \widehat\bbeta(\lambda^-))| = \lambda^-$ can
only hold with zero probability (see
e.g.~\cite{tibshirani2013Lasso}). Hence, at least one of the following
statements hold:
\begin{itemize}
\item[(a)] $|\bX^{\transp}_j(\by - \bX\widehat\bbeta(\lambda^-))| = \lambda^-|g_j(\lambda^-)| \in \left[ (1 - a)\lambda^-, \lambda^- \right)$;
\item[(b)] 
$|\bX_j^{\transp} \bm u |  \ge \frac{\lambda - (1-a)\lambda^- }{\| \bX (\widehat\bbeta(\lambda) - \widehat\bbeta(\lambda^-))\|} > \frac{ a\lambda^- }{\| \bX (\widehat\bbeta(\lambda) - \widehat\bbeta(\lambda^-))\|}$.
\end{itemize}
In the second case,
since the spectral norm $\sigma_{\max}(\bX)$ is bounded in probability
(see e.g.~\cite{vershynin}), we make use of Lemma
\ref{lm:l2_local_continu} to conclude that
\[
\frac{ a\lambda^- }{\| \bX (\widehat\bbeta(\lambda) - \widehat\bbeta(\lambda^-))\|} \ge \frac{ a\lambda^- }{\sigma_{\max}(\bX) \|\widehat\bbeta(\lambda) - \widehat\bbeta(\lambda^-) \|} \ge \frac{ a\lambda^- }{c \sigma_{\max}(\bX)\sqrt{(\lambda^+ - \lambda^-)p} } \ge c'a\sqrt{\frac{m}{n}}
\]
holds for all $\lambda^- \le \lambda \le \lambda^+$ with probability tending to one. Above, the constant $c'$ only depends on $\lambda_{\min}, \lambda_{\max}, \delta$ and $\Theta$. Consequently, we see that
\begin{multline}\nonumber
\sup_{\lambda^- \le \lambda \le \lambda^+}\left| \mathcal{S}(\lambda) \setminus \mathcal{S}(\lambda^-) \right| \le \#\left\{j: (1-a)\lambda^- \le |\bX_j^{\transp}(\by - \bX\widehat\bbeta(\lambda^-))| < \lambda^- \right\} \\
+ \#\left\{ j: |\bX^{\transp}_j \bm u| > c'a\sqrt{m/n} \right\}.
\end{multline}
Equality (3.21) of \cite{BM12} guarantees the existence of a constant $a$ such that the event\footnote{Apply Theorem 1.8 to carry over the results for AMP iterates to Lasso solution.}
\begin{equation}\label{eq:apply_continu_event1}
 \#\left\{j: (1-a)\lambda^- \le |\bX_j^{\transp}(\by - \bX\widehat\bbeta(\lambda^-))| < \lambda^- \right\} \le \frac{\omega p}{4}
\end{equation}
happens with probability approaching one. Since
$\lambda^- = \lambda_i$ is always in the interval $[\lambda_{\min},
\lambda_{\max}]$, the constant
$a$
can be made to be independent of the index $i$.
For the second term, it follows from Lemma \ref{lm:weak_corre} that
for sufficiently large $m$, the event
\begin{equation}\label{eq:apply_weak_event2}
\#\left\{ j: |\bX^{\transp}_j \bm u| > c'a\sqrt{m/n} \right\} \le \frac{\omega p}{4}
\end{equation}
also holds with probability approaching one. Combining
\eqref{eq:apply_continu_event1} and \eqref{eq:apply_weak_event2}, we
get
\begin{equation}\label{eq:support_continu}
\sup_{\lambda^- \le \lambda \le \lambda^+}\left| \mathcal{S}(\lambda) \setminus \mathcal{S}(\lambda^-) \right| \le \frac{\omega p}{2}
\end{equation}
holds with probability tending to one. 

Next, we bound
$| \mathcal{S}(\lambda^-) \setminus \mathcal{S}(\lambda)|$. Applying
Theorem 1.5 in \cite{BM12}, we can find a constant $\nu > 0$
independent of $\lambda^- \in [\lambda_{\min}, \lambda_{\max}]$ such
that
\begin{equation}\label{eq:small_beta_even}
\# \left\{ j: 0 < |\widehat\beta_j(\lambda^-)| < \nu \right\} \le \frac{\omega p}{4}
\end{equation}
happens with probability approaching one. Furthermore, the simple inequality
\[
\|\widehat\bbeta(\lambda) - \widehat\bbeta(\lambda^-)\|^2 \ge \nu^2 \# \left\{ j: j\in \mathcal{S}(\lambda^-) \setminus \mathcal{S}(\lambda), |\widehat\beta_j(\lambda^-)| \ge \nu \right\},
\]
together with Lemma \ref{lm:l2_local_continu},
give 
\begin{equation}\label{eq:small_beta_big}
\# \left\{ j: j\in \mathcal{S}(\lambda^-) \setminus \mathcal{S}(\lambda), |\widehat\beta_j(\lambda^-)| \ge \nu \right\} \le \frac{\|\widehat\bbeta(\lambda) - \widehat\bbeta(\lambda^-)\|^2}{\nu^2} \le \frac{c^2(\lambda^+-\lambda^-) p}{\nu^2}
\end{equation}
for all $\lambda \in [\lambda_{\min}, \lambda_{\max}]$ with
probability converging to one. Taking $m$ sufficiently large such that
$\lambda^+ - \lambda^- = (\lambda_{\max} - \lambda_{\min})/m \le
\omega \nu^2/4c^2$
in \eqref{eq:small_beta_big} and combining this with
\eqref{eq:small_beta_even} gives that 
\begin{equation}\label{eq:support_continu2}
\sup_{\lambda^- \le \lambda \le \lambda^+}\left| \mathcal{S}(\lambda^-) \setminus \mathcal{S}(\lambda) \right| \le \frac{\omega p}{2}
\end{equation}
holds with probability tending to one.

To conclude the proof, note that both \eqref{eq:support_continu} and
\eqref{eq:support_continu2} hold for a large but fixed
$m$. Substituting these two inequalities into \eqref{eq:delta_set}
confirms \eqref{eq:support_const} by taking a union bound. 

As far as the true discovery number $T(\lambda)$ is concerned, all the
arguments seamlessly apply and we do not repeat them. This terminates
the proof.
\end{proof}


\subsection{Proofs of auxiliary lemmas}
\label{sec:proof-auxil-lemm}
In this section, we prove Lemmas~\ref{lm:weak_corre} and
\ref{lm:l2_local_continu}. While the proof of the first is
straightforward, the second crucially relies on
Lemma~\ref{lm:uniform_l2norm}, whose proof makes use of Lemmas~\ref{lm:reverse_cauchy} and \ref{lm:small_sing_ld}. Hereafter,
we denote by $o_{\P}(1)$ any random variable which tends to zero in
probability.

\begin{proof}[Proof of Lemma~\ref{lm:weak_corre}]
Since
\[
\|\bm u^\T \bX\| ^2 \ge \frac{r^2}{n}\#\left\{1 \le j \le p: |\bX_j^{\transp} \bm u| > \frac{r}{\sqrt{n}} \right\},
\]
we have 
\[
\begin{aligned}
\#\left\{1 \le j \le p: |\bX_j^{\transp} \bm u| > \frac{r}{\sqrt{n}} \right\} &\le \frac{n}{r^2}\|\bm u^\T \bX\|^2  \le \frac{n\sigma_{\max}(\bX)^2 \|\bm u\|^2}{r^2} \\ & = (1+o_{\P}(1))\frac{(1 + \sqrt{\delta})^2 p}{r^2},
\end{aligned}
\]
where we make use of $\lim n/p = \delta$ and
$\sigma_{\max}(\bX) = 1 + \delta^{-1/2} + o_{\P}(1)$. To complete the
proof, take any $r_c > 0$ such that
$(1 + \sqrt{\delta})/r_c < \sqrt{c}$.

\end{proof}


\begin{lemma}\label{lm:reverse_cauchy}
  Take a sequence $a_1 \ge a_2 \ge \cdots \ge a_p \ge 0$ with at least
  one strict inequality, and suppose that
\begin{equation}\nonumber
\frac{p\sum_{i=1}^pa_i^2}{\left(\sum_{i=1}^p a_i \right)^2} \ge M
\end{equation}
for some $M > 1$. Then for any $1 \le s \le p$,
\[
\frac{\sum_{i=1}^{s}a_i^2}{\sum_{i=1}^p a_i^2} \ge 1 - \frac{p^3}{M s^3}.
\]
\end{lemma}

\begin{proof}[Proof of Lemma \ref{lm:reverse_cauchy}]
By the monotonicity of $\bm a$, 
\[
\frac{\sum_{i=1}^s a_i^2}{s} \ge \frac{\sum_{i=1}^p a_i^2}{p},
\]
which implies
\begin{equation}\label{eq:sq_ratio}
\frac{p\sum_{i=1}^s a_i^2}{\left(\sum_{i=1}^s a_i \right)^2} \ge \frac{s\sum_{i=1}^p a_i^2}{\left(\sum_{i=1}^p a_i \right)^2} \ge \frac{s M}{p}.
\end{equation}
Similarly,
\begin{equation}\label{eq:sq_big}
\sum_{i=s+1}^p a_i^2 \le (p-s)\left( \frac{\sum_{i=1}^s a_i}{s} \right)^2
\end{equation}
and it follows from \eqref{eq:sq_ratio} and \eqref{eq:sq_big} that 
\[
\frac{\sum_{i=1}^s a_i^2}{\sum_{i=1}^p a_i^2} \ge \frac{\sum_{i=1}^s a_i^2}{\sum_{i=1}^s a_i^2 + (p-s)\left( \frac{\sum_{i=1}^s a_i}{s} \right)^2} \ge \frac{\frac{sM}{p^2}}{\frac{sM}{p^2} + \frac{p-s}{s^2}} \ge 1 - \frac{p^3}{Ms^3}.
\]
\end{proof}

\begin{lemma}\label{lm:small_sing_ld}
  Assume $n/p \goto 1$, i.e.~$\delta = 1$. Suppose $s$ obeys
  $s/p \goto 0.01$. Then with probability tending to one, the smallest singular value obeys 
\[
\min_{|{\mathcal{S}}| = s} ~ \sigma_{\min}(\bX_{\mathcal{S}}) \ge \frac12,
\]
where the minimization is over all subsets of $\{1, \ldots, p\}$ of
cardinality $s$.
\end{lemma}

\begin{proof}[Proof of Lemma~\ref{lm:small_sing_ld}]
  For a fixed ${\mathcal{S}}$, we have
\[
\P\left( \sigma_{\min}(\bX_{\mathcal{S}}) < 1 - \sqrt{s/n} - t \right) \le \e^{-nt^2/2}
\]
for all $t \ge 0$, please see \cite{vershynin}. The claim follows from
plugging $t = 0.399$ and a union bound over
${p \choose s} \le \exp(p H(s/p))$ subsets, where
$H(q) = -q\log q - (1-q) \log(1-q)$. We omit the details.
\end{proof}

The next lemma bounds the Lasso solution in $\ell_2$ norm uniformly over $\lambda$. We use some ideas from the proof of Lemma 3.2 in \cite{BM12}.
\begin{lemma}\label{lm:uniform_l2norm}
Given any positive constants $\lambda_{\min} < \lambda_{\max}$, there exists a constant $C$ such that
\[
\P\left( \sup_{\lambda_{\min} \le \lambda \le \lambda_{\max}} \|\widehat{\bm\beta}(\lambda)\| \le C\sqrt{p} \right) \goto 1.
\]
\end{lemma}

\begin{proof}[Proof of Lemma \ref{lm:uniform_l2norm}]
  For simplicity, we omit the dependency of $\widehat\bbeta$ on
  $\lambda$ when clear from context. We first consider the case where
  $\delta < 1$. Write
  $\widehat\bbeta = \mathcal{P}_1(\widehat\bbeta) +
  \mathcal{P}_2(\widehat\bbeta)$,
  where $\mathcal{P}_1(\widehat\bbeta)$ is the projection of
  $\widehat\bbeta$ onto the null space of $\bX$ and
  $\mathcal{P}_2(\widehat\bbeta)$ is the projection of
  $\widehat\bbeta$ onto the row space of $\bX$. By the rotational
  invariance of \iid~Gaussian vectors, the null space of $\bX$ is a
  random subspace of dimension $p - n = (1 - \delta + o(1))p$ with
  uniform orientation. Since $\mathcal{P}_1(\widehat\bbeta)$ belongs
  to the null space, Kashin's Theorem (see Theorem F.1 in \cite{BM12})
  gives that with probability at least $1 - 2^{-p}$,
\begin{equation}\label{eq:beta_decompose}
\begin{aligned}
\|\widehat\bbeta\|^2 &= \|\mathcal{P}_1(\widehat\bbeta)\|^2 + \|\mathcal{P}_2(\widehat\bbeta)\|^2 \\
&\le c_1\frac{\|\mathcal{P}_1(\widehat\bbeta)\|_1^2}{p} + \|\mathcal{P}_2(\widehat\bbeta)\|^2\\
                     & \le 2c_1\frac{\|\widehat\bbeta\|_1^2+ \|\mathcal{P}_2(\widehat\bbeta)\|_1^2}{p} + \|\mathcal{P}_2(\widehat\bbeta)\|^2\\
                     & \le \frac{2c_1\|\widehat\bbeta\|_1^2}{p} + (1+2c_1)\|\mathcal{P}_2(\widehat\bbeta)\|^2
\end{aligned}
\end{equation}
for some constant $c_1$ depending only on $\delta$; the first step
uses Kashin's theorem, the second the triangle inequality, and the
third Cauchy-Schwarz inequality.  The smallest nonzero singular value
of the Wishart matrix $\bX^\T \bX$ is concentrated at
$(1/\sqrt{\delta} - 1)^2$ with probability tending to one (see
e.g.~\cite{vershynin}). In addition, since
$\mathcal{P}_2(\widehat\bbeta)$ belongs to the row space of $\bX$, we
have
\[
\|\mathcal{P}_2(\widehat\bbeta)\|^2 \le c_2 \|\bX \mathcal{P}_2(\widehat\bbeta)\|^2
\]
with probability approaching one. Above, $c_2$ can be chosen to be
$(1/\sqrt{\delta} - 1)^{-2} + o(1)$. Set $c_3 =
c_2(1+2c_1)$. Continuing \eqref{eq:beta_decompose} yields
\[
\begin{aligned}
\|\widehat\bbeta\|^2 &\le \frac{2c_1\|\widehat\bbeta\|_1^2}{p} + c_2(1+2c_1)\|\bX\mathcal{P}_2(\widehat\bbeta)\|^2 \\
                        &= \frac{2c_1\|\widehat\bbeta\|_1^2}{p} + c_3\| \bX \widehat\bbeta\|^2\\
                        &\le \frac{2c_1\|\widehat\bbeta\|_1^2}{p} + 2c_3\|\by - \bX \widehat\bbeta\|^2 + 2c_3\|\by\|^2\\
                        &\le \frac{2c_1\left( \frac12\|\by- \bX\widehat\bbeta\|^2 + \lambda\|\widehat\bbeta\|_1 \right)^2}{\lambda^2p} + 4c_3\left( \frac12\|\by - \bX\widehat\bbeta\|^2+\lambda\|\widehat\bbeta\|_1\right) + 2c_3\|\by\|^2\\
                        &\le \frac{c_1\|\by\|^4}{2\lambda^2p} + 4c_3\|\by\|^2,
\end{aligned}
\]
where in the last inequality we use the fact
$\frac12 \|\by-\bX \widehat\bbeta\|^2 + \lambda\|\widehat\bbeta\|_1
\le \frac12 \|\by\|^2$.
Thus, it suffices to bound $\|\by\|^2$. The largest singular value of
$\bX^\T \bX$ is bounded above by
$(1/\sqrt{\delta} +1)^2 + o_{\P}(1)$. Therefore, 
\[
\|\by\|^2 = \|\bX \bbeta + \bz\|^2 \le 2\|\bX\bbeta\|^2 + 2\|\bz\|^2 \le c_4\|\bbeta\|^2 + 2\|\bz\|^2.
\]
Since both $\beta_i$ and $z_i$ have bounded second moments, the law of
large numbers claims that there exists a constant $c_5$ such that
$c_4\|\bbeta\|^2 + 2\|\bz\|^2 \le c_5 p$ with probability approaching
one. Combining all the inequalities above gives
\[
\sup_{\lambda_{\min} \le \lambda \le \lambda_{\max}}\|\widehat\bbeta(\lambda)\|^2 \le \frac{c_1c_5^2p}{2\lambda^2} + 2c_3c_5p \le \left( \frac{c_1c_5^2}{2\lambda_{\min}^2} + 2c_3c_5 \right)p
\]
with probability converging to one.

In the case where $\delta > 1$, the null space of $\bX$ reduces to
$\mathbf{0}$, hence $\mathcal{P}_1(\widehat\bbeta) = 0$. Therefore,
this reduces to a special case of the above argument.

Now, we turn to work on the case where $\delta = 1$. We start with
\[
\| \bX \widehat\bbeta \|^2 \le 2 \|\by\|^2 + 2\|\by-\bX \widehat\bbeta\|^2
\]
and
\[
\frac12 \|\by-\bX \widehat\bbeta\|^2 + \lambda\|\widehat\bbeta\|_1 \le \frac12 \|\by\|^2.
\]
These two inequalities give that simultaneously over all $\lambda$,
\begin{subequations}\label{eq:xbeta_l1}
\begin{align}
&\| \bX \widehat\bbeta(\lambda) \|^2 \le 4 \|\by\|^2 \le 4c_5 p \label{eq:xbeta}\\
& \|\widehat\bbeta(\lambda)\|_1 \le \frac1{2\lambda_{\min}} \|\by\|^2 \le \frac{c_5 p}{2\lambda_{\min}}\label{eq:l1bd}
\end{align}
\end{subequations}
with probability converging to one.
Let $M$ be any constant larger than $1.7 \times 10^7$. If $ p \|\widehat\bbeta\|^2/\|\widehat\bbeta\|_1^2 < M$, by \eqref{eq:l1bd}, we get
\begin{equation}\label{eq:beta_bound_d1}
\|\widehat\bbeta\| \le \frac{\sqrt{M} c_5}{2\lambda_{\min}} \sqrt{p}.
\end{equation}
Otherwise, denoting by ${\mathcal{T}}$ the set of indices $1 \le i \le p$ that correspond to the $s := \lceil p/100 \rceil$ largest $|\widehat\beta|_i$, from Lemma~\ref{lm:reverse_cauchy} we have
\begin{equation}\label{eq:lemma_m}
\frac{\|\widehat\bbeta_{{\mathcal{T}}}\|^2}{\|\widehat\bbeta\|^2} \ge 1 - \frac{p^3}{M s^3} \ge 1 - \frac{10^6}{M}.
\end{equation}
To proceed, note that
\begin{equation}\nonumber
\begin{aligned}
\| \bX \widehat\bbeta \| &= \| \bX_{\mathcal{T}} \widehat\bbeta_{\mathcal{T}} + \bX_{\overline {\mathcal{T}}} \widehat\bbeta_{\overline {\mathcal{T}}}\| \\
&\ge \| \bX_{\mathcal{T}} \widehat\bbeta_{\mathcal{T}}\| - \|\bX_{\overline {\mathcal{T}}} \widehat\bbeta_{\overline {\mathcal{T}}}\| \\
&\ge \| \bX_{\mathcal{T}} \widehat\bbeta_{\mathcal{T}}\| - \sigma_{\max}(\bX) \|\widehat\bbeta_{\overline {\mathcal{T}}}\|.
\end{aligned}
\end{equation}
By Lemma~\ref{lm:small_sing_ld}, we get
$\| \bX_{\mathcal{T}} \widehat\bbeta_{\mathcal{T}}\| \ge \frac12 \|\bbeta_{\mathcal{T}}\|$, and it is also
clear that $\sigma_{\max}(\bX) = 2 + o_{\P}(1)$. Thus, by
\eqref{eq:lemma_m} we obtain
\[
\begin{aligned}
\| \bX \widehat\bbeta \| \ge \| \bX_{\mathcal{T}} \widehat\bbeta_{\mathcal{T}}\| - \sigma_{\max}(\bX) \|\widehat\bbeta_{\overline {\mathcal{T}}}\|
&\ge \frac12 \| \widehat\bbeta_{\mathcal{T}}\| - (2 + o_{\P}(1))\|\widehat\bbeta_{\overline {\mathcal{T}}}\|\\
&\ge \frac12 \sqrt{1 - \frac{10^6}{M}} \|\widehat\bbeta\| - (2 + o_{\P}(1)) \sqrt{\frac{10^6}{M}} \|\widehat\bbeta\|\\
& = (c + o_{\P}(1))\|\widehat\bbeta\|,
\end{aligned}
\]
where $c = \frac12 \sqrt{1 - 10^6/M} - 2\sqrt{10^6/M} > 0$. Hence, owing to \eqref{eq:xbeta},
\begin{equation}\label{eq:beta_bound_d2}
\|\widehat\bbeta\| \le \frac{(2+o_{\P}(1))\sqrt{c_5}}{c} \sqrt{p}
\end{equation}

In summary, with probability tending to one, in either case, namely, \eqref{eq:beta_bound_d1} or \eqref{eq:beta_bound_d2},
\[
\|\widehat\bbeta\| \le C\sqrt{p}
\]
for come constant $C$. This holds uniformly for all $\lambda \in [\lambda_{\min}, \lambda_{\max}]$, and thus completes the proof.

\end{proof}


Now, we conclude this section by proving our main lemma.
\begin{proof}[Proof of Lemma \ref{lm:l2_local_continu}]
The proof extensively applies Lemma 3.1\footnote{The conclusion of this lemma, $\|\bm r\| \le \sqrt{p}\xi(\epsilon, c_1, \ldots, c_5)$ in our notation, can be effortlessly strengthened to $\|\bm r\| \le (\sqrt{\epsilon} + \epsilon/\lambda)\xi(c_1, \ldots, c_5)\sqrt{p}$.} in \cite{BM12} and Lemma~\ref{lm:uniform_l2norm}. Let $\bm x + \bm r = \widehat\bbeta(\lambda)$ and $\bm x = \widehat\bbeta(\lambda^-)$ be the notations in the statement of Lemma 3.1 in \cite{BM12}. Among the fives assumptions needed in that lemma, it suffices to verify the first, third and fourth. Lemma \ref{lm:uniform_l2norm} asserts that
\[
\sup_{\lambda^-\le \lambda \le \lambda^+} \|\bm r(\lambda)\| = \sup_{\lambda^-\le \lambda \le \lambda^+}\|\widehat\bbeta(\lambda) - \widehat\bbeta(\lambda^-)\| \le 2A\sqrt{p}
\]
with probability approaching one. This fulfills the first assumption by taking $c_1 = 2A$. Next, let $\bg(\lambda^-) \in \partial\|\widehat\bbeta(\lambda^-)\|_1$ obey
\[
\bX^\T(\by - \bX \widehat\bbeta(\lambda^-)) = \lambda^- \bg(\lambda^-).
\]
Hence,
\[
\left\| - \bX^\T(\by - \bX\widehat\bbeta(\lambda^-)) + \lambda \bg(\lambda^-)  \right\| = (\lambda - \lambda^-)\|\bg(\lambda^-)\| \le (\lambda^+ - \lambda^-)\sqrt{p}
\]
which certifies the third assumption. To verify the fourth assumption,
taking $t \goto \infty$ in Proposition 3.6 of \cite{BM12} ensures the
existence of constants $c_2, c_3$ and $c_4$ such that, with
probability tending to one, $\sigma_{\min}(\bX_{T\cup T'}) \ge c_4$
for $T = \{ j: |g_j(\lambda^-)| \ge 1 - c_2\}$ and arbitrary
$T' \subset \{1, \ldots, p\}$ with $|T'| \le c_3 p$. Further, these
constants can be independent of $\lambda^-$ since
$\lambda^- \in [\lambda_{\min}, \lambda_{\max}]$ belongs to a
compact interval. Therefore, this lemma concludes that, with probability
approaching one,
\begin{align*}
\sup_{\lambda^- \le \lambda \le \lambda^+} \| \widehat\bbeta(\lambda) - \widehat\bbeta(\lambda^-)\| &\le  \sup_{\lambda^- \le \lambda \le \lambda^+}  \left(\sqrt{\lambda-\lambda^-} + \frac{\lambda-\lambda^-}{\lambda} \right) \xi \sqrt{p} \\
&\le \left(\sqrt{\lambda^+-\lambda^-} + \frac{\lambda^+-\lambda^-}{\lambda_{\min}} \right) \xi \sqrt{p} \\
&= O\left(\sqrt{(\lambda^+-\lambda^-)p}\right).
\end{align*}
This finishes the proof.
\end{proof}

\section{Optimizing the Trade-off}
\label{sec:optimizing-tradeoff}

In this section, we still work under our working hypothesis and
$\sigma > 0$. Fixing $\delta$ and $\epsilon$, we aim to show that no pairs below the boundary curve can be realized. 
Owing to the uniform convergence established in
Appendix~\ref{sec:uniform-over-lambda}, it is sufficient to study the
range of $(\fpower(\lambda), \ffdp(\lambda))$ introduced in Appendix \ref{sec:proof-roadmap} by varying $\Pi^\star$
and $\lambda$. To this end, we introduce a useful trick based on the
following lemma.

\begin{lemma}\label{lm:convexity}
For any fixed $\alpha > 0$, define a function $y = f(x)$ in the parametric form:
\begin{align*}
x(t)& = \P(|t + W| > \alpha)\\
y(t)& = \E \left( \eta_{\alpha}(t + W) - t \right)^2
\end{align*}
for $t \ge 0$, where $W$ is a standard normal. Then $f$ is strictly
concave.
\end{lemma}

We use this to simplify the problem of detecting feasible pairs
$(\fpower, \ffdp)$. Denote by $\pi^\star := \Theta^\star/\tau$. Then
\eqref{eq:amp_eqn} implies
\begin{equation}\label{eq:eqn_temp}
\begin{aligned}
(1-\epsilon)\E\eta_{\alpha}(W)^2 + \epsilon \E(\eta_{\alpha}(\pi^{\star} + W) - \pi^{\star})^2 & < \delta,\\
(1-\epsilon)\P(|W| > \alpha) + \epsilon \P(|\pi^{\star} + W| > \alpha) & < \min\{\delta, 1\}. 
\end{aligned}
\end{equation}
We emphasize that \eqref{eq:eqn_temp} is not only necessary, but also sufficient in the following sense: given $0 < \delta_1 < \delta, 0 < \delta_2 < \min\{\delta, 1\}$ and $\pi^\star$, we can solve for $\tau$ by setting
\[
(1-\epsilon)\E\eta_{\alpha}(W)^2 + \epsilon \E(\eta_{\alpha}(\pi^{\star} + W) - \pi^{\star})^2 = \delta_1
\]
and making use of the first line of \eqref{eq:amp_eqn}, which can be alternatively written as
\[
\tau^2 = \sigma^2 + \frac{\tau^2}{\delta} \left[ (1-\epsilon)\E\eta_{\alpha}(W)^2 + \epsilon \E(\eta_{\alpha}(\pi^{\star} + W) - \pi^{\star})^2 \right].
\]
($\Pi^\star = \tau \pi^\star$ is also determined, so does
$\Pi$),\footnote{Not every pair
  $(\delta_1, \delta_2) \in (0, \delta) \times (0, \min\{\delta, 1\})$
  is feasible below the Donoho-Tanner phase transition. Nevertheless, this does not affect our discussion.} and along
with
\[
(1-\epsilon)\P(|W| > \alpha) + \epsilon \P(|\pi^{\star} + W| > \alpha) = \delta_2,
\]
$\lambda$ is also uniquely determined. 

Since \eqref{eq:eqn_temp} is invariant if $\pi^\star$ is replaced by
$|\pi^\star|$, we assume $\pi^\star \ge 0$ without loss of
generality. As a function of $t$, $\P(|t + W| > \alpha)$ attains the
minimum $\P(|W|> \alpha) = 2\Phi(-\alpha)$ at $t = 0$, and the
supremum equal to one at $t = \infty$. Hence, there must exist
$\epsilon' \in (0, 1)$ obeying
\begin{equation}\label{eq:intro_epsi_prime}
\P(|\pi^{\star} + W| > \alpha)  = (1 - \epsilon')\P(|W| > \alpha) + \epsilon'.
\end{equation}
As a consequence, the predicted TPP and FDP can be alternatively expressed as
\begin{equation}\label{eq:fdr_power_compact}
\fpower = 2(1-\epsilon')\Phi(-\alpha) + \epsilon', \quad \ffdp = \frac{2(1-\epsilon)\Phi(-\alpha)}{2(1-\epsilon\epsilon')\Phi(-\alpha) + \epsilon\epsilon'}.
\end{equation}
Compared to the original formulation, this expression is preferred
since it only involves scalars.

Now, we seek equations that govern $\epsilon'$ and $\alpha$, given
$\delta$ and $\epsilon$. Since both $\E (\eta_{\alpha}(t + W) - t)^2$
and $\P(|t + W| > \alpha)$ are monotonically increasing with respect
to $t \ge 0$, there exists a function $f$ obeying
\[
\E (\eta_{\alpha}(t + W) - t)^2 = f\left( \P(|t + W| > \alpha) \right).
\]
Lemma \ref{lm:convexity} states that $f$ is concave.  Then 
\[
  \E(\eta_{\alpha}(\pi^{\star} + W) - \pi^{\star})^2 = 
 f \left( \P( |\pi^\star + W| > \alpha \right)
\]
and \eqref{eq:intro_epsi_prime} allows us to view the argument of $f$ in the right-hand side as an average of a random variable taking value  $\P(|W|  >  \alpha)$ with probability  $1-\epsilon'$  and value one with probability $\epsilon'$. Therefore, Jensen's inequality states that 
\[
  \E(\eta_{\alpha}(\pi^{\star} + W) - \pi^{\star})^2 
   \ge (1 - \epsilon')f \left( \P(|W|  >  \alpha) \right)+ \epsilon'f(1)\\
   = (1-\epsilon') \E \eta_{\alpha}(W)^2 + \epsilon'(\alpha^2 + 1).
\]
Combining this with \eqref{eq:eqn_temp} gives
\begin{subequations}\label{eq:two_eqn_comb}
\begin{align}
&(1-\epsilon\epsilon')\E\eta_{\alpha}(W)^2 + \epsilon\epsilon' (\alpha^2 + 1)   < \delta, \label{eq:amp_restrict_w}\\ 
&(1-\epsilon\epsilon')\P(|W| > \alpha) + \epsilon\epsilon'   < \min\{\delta, 1\}. \label{eq:amp_restrict_w2}
\end{align}
\end{subequations}
Similar to \eqref{eq:eqn_temp}, \eqref{eq:two_eqn_comb} is also sufficient in the same sense, and \eqref{eq:amp_restrict_w2} is automatically satisfied if $\delta > 1$.

The remaining part of this section studies the range of
$(\fpower, \ffdp)$ given by \eqref{eq:fdr_power_compact} under the
constraints \eqref{eq:two_eqn_comb}. Before delving into the details,
we remark that this reduction of $\pi^\star$ to a two-point prior is
realized by setting $\pi^\star = \infty$ 
(equivalently $\Pi^\star = \infty$) with probability $\epsilon'$ and
otherwise $+0$, where $+0$ is considered to be nonzero. Though this
prior is not valid since the working hypothesis requires a finite
second moment, it can nevertheless be approximated by a sequence of
instances, please see the example given in
Section~\ref{sec:main-results}.

The lemma below recognizes that for certain $(\delta, \epsilon)$
pairs, the TPP is asymptotically bounded above away from 1.
\begin{lemma}\label{lm:gamma_max}
Put
\[
 u^\star(\delta, \epsilon) :=
\begin{cases}
1 - \frac{(1-\delta)(\epsilon-\epsilon^\star)}{\epsilon(1-\epsilon^\star)},  & \quad \delta < 1 \mbox{ and } \epsilon > \epsilon^\star(\delta), \\
1,  & \quad \mbox{otherwise}.
\end{cases}
\]
Then
\[
\fpower <  u^\star(\delta, \epsilon).
\]
Moreover, $\fpower$ can be arbitrarily close to $ u^\star$. 
\end{lemma}
This lemma directly implies that above the Donoho-Tanner phase
transition (i.e. $\delta < 1$ and
$\epsilon > \epsilon^\star(\delta)$), there is a fundamental limit on
the TPP for arbitrarily strong signals. Consider
\begin{equation}\label{eq:phase_cri}
2(1-\epsilon)\left[ (1+t^2)\Phi(-t) - t\phi(t) \right] + \epsilon(1 + t^2) = \delta.
\end{equation}
For $\delta < 1$, $\epsilon^\star$ is the only positive constant in $(0, 1)$ such that \eqref{eq:phase_cri} with $\epsilon = \epsilon^\star$ has a unique positive root. Alternatively, the function $\epsilon^\star = \epsilon^\star(\delta)$ is implicitly given in the following parametric form:
\begin{align*}
&\delta = \frac{2\phi(t)}{2\phi(t) + t(2\Phi(t) - 1)}\\
&\epsilon^\star = \frac{2\phi(t) - 2t \Phi(-t) }{2\phi(t) + t(2\Phi(t) - 1)}
\end{align*}
for $t > 0$, from which we see that $\epsilon^\star < \delta < 1$. Take
the sparsity level $k$ such as
$\epsilon^\star p < k < \delta p = n$, from which we have $ u^\star < 1$.
As a result, the Lasso is unable to select all the $k$ true signals
even when the signal strength is arbitrarily high. This is happening even though the Lasso has the chance to select up to $n > k$ variables.

Any $u$ between 0 and $ u^\star$ (non-inclusive) can be realized as
$\fpower$. Recall that we denote by $t^\star(u)$ the unique root in
$(\alpha_0, \infty)$ ($\alpha_0$ is the root of $(1+t^2)\Phi(-t)
-t\phi(t) = \delta/2$) to the equation
\begin{equation}\label{eq:gamma_sandwich}
\frac{2(1 - \epsilon)\left[ (1+t^2)\Phi(-t) - t\phi(t) \right] + \epsilon(1 + t^2) - \delta}{\epsilon\left[ (1+t^2)(1-2\Phi(-t)) + 2t\phi(t) \right]} = \frac{1 - u}{1 - 2\Phi(-t)}.
\end{equation}
 For a proof of this fact, we refer to Lemma~\ref{lm:root_exist}. Last, recall that
\begin{equation}\nonumber
\mfdr(u; \delta, \epsilon) = \frac{2(1-\epsilon)\Phi(-t^\star(u))}{2(1-\epsilon)\Phi(-t^\star(u)) + \epsilon u}.
\end{equation}

We can now state the fundamental trade-off between $\ffdp$ and
$\fpower$.
\begin{lemma}\label{lm:alpha_large_gm}
If $\fpower \ge u$ for $u \in (0,  u^\star)$, then
\[
\ffdp > \mfdr(u).
\]
In addition, $\ffdp$ can be arbitrarily close to $q^\star(u)$.
\end{lemma}



\subsection{Proofs of Lemmas~\ref{lm:convexity}, \ref{lm:gamma_max} and \ref{lm:alpha_large_gm}}
\label{sec:proofs-lemm-reflm:g}
\begin{proof}[Proof of Lemma~\ref{lm:convexity}]
  First of all, $f$ is well-defined since both $x(t)$ and $y(t)$ are
  strictly increasing functions of $t$. Note that
\[
\frac{\diff x}{\diff t} = \phi(\alpha-t) - \phi(-\alpha-t), \quad \frac{\diff y}{\diff t} = 2t\left[ \Phi(\alpha-t) - \Phi(-\alpha-t) \right].
\]
Applying the chain rule gives 
\begin{equation}\nonumber
\begin{aligned}    
f'(t) &= \frac{\diff y}{\diff t} \Big/ \frac{\diff x}{\diff t} = \frac{2t\left[ \Phi(\alpha-t) - \Phi(-\alpha-t) \right]}{\phi(\alpha-t) - \phi(-\alpha-t)}\\
&= \frac{2t \int_{-\alpha-t}^{\alpha-t} \e^{-\frac{u^2}{2}} \diff u}{\e^{-\frac{(\alpha-t)^2}{2}} - \e^{-\frac{(-\alpha-t)^2}{2}}}
= \frac{2t \int_{-\alpha}^{\alpha} \e^{-\frac{(u-t)^2}{2}} \diff u}{\e^{-\frac{\alpha^2+t^2-2\alpha t}{2}} - \e^{-\frac{\alpha^2+t^2 + 2\alpha t}{2}}}\\
&= \frac{2t \e^{\frac{\alpha^2}{2}}\int_{-\alpha}^{\alpha} \e^{-\frac{u^2}{2}} \e^{t u}\diff u}{\e^{\alpha t} - \e^{-\alpha t}} = \frac{2 \e^{\frac{\alpha^2}{2}}\int_0^{\alpha} \e^{-\frac{u^2}{2}} (\e^{t u} + \e^{-t u} ) \diff u}{\int^{\alpha}_0 \e^{t u} + \e^{-t u} \diff u}.
\end{aligned}
\end{equation}
Since $x(t)$ is strictly increasing in $t$, we see that $f''(t) \le 0$
is equivalent to saying that the function
\[
g(t) := \frac{\int_0^{\alpha} \e^{-\frac{u^2}{2}} (\e^{t u} + \e^{-t u} ) \diff u}{\int^{\alpha}_0 \e^{t u} + \e^{-t u} \diff u} \equiv \frac{\int_0^{\alpha} \e^{-\frac{u^2}{2}} \cosh(t u) \diff u}{\int^{\alpha}_0 \cosh(t u) \diff u}
\]
is decreasing in $t$. Hence, it suffices to show that 
\[
g'(t) = \frac{\int_0^{\alpha} \e^{-\frac{u^2}{2}} u\sinh(tu) \diff u \int^{\alpha}_0 \cosh(tv) \diff v - \int_0^{\alpha} \e^{-\frac{u^2}{2}} \cosh(tu) \diff u \int^{\alpha}_0 v \sinh(tv) \diff v}{\left( \int^{\alpha}_0 \cosh(t v) \diff v\right)^2} \le 0.
\]
Observe that the numerator is equal to 
\[
\begin{aligned}
\int_0^{\alpha} \int_0^{\alpha} \e^{-\frac{u^2}{2}} & u \sinh(tu) \cosh(tv)\diff u \diff v - \int_0^{\alpha}\int_0^{\alpha} \e^{-\frac{u^2}{2}} v \cosh(tu) \sinh(tv) \diff u \diff v \\
& \,\, = \int_0^{\alpha} \int_0^{\alpha} \e^{-\frac{u^2}{2}} \left( u \sinh(tu) \cosh(tv) - v \cosh(tu) \sinh(tv)  \right) \diff u \diff v\\
& \overset{u \leftrightarrow v}{=} \int_0^{\alpha} \int_0^{\alpha} \e^{-\frac{v^2}{2}} \left( v \sinh(tv) \cosh(tu) - u \cosh(tv) \sinh(tu) \right) \diff v \diff u \\
& \,\, = \frac12 \int_0^{\alpha} \int_0^{\alpha} (\e^{-\frac{u^2}{2}} - \e^{-\frac{v^2}{2}})\left( u \sinh(tu) \cosh(tv) - v \cosh(tu) \sinh(tv)  \right) \diff u \diff v.
\end{aligned}
\]
Then it is sufficient to show that
\[
(\e^{-\frac{u^2}{2}} - \e^{-\frac{v^2}{2}})\left( u \sinh(tu) \cosh(tv) - v \cosh(tu) \sinh(tv)  \right) \le 0
\]
for all $u, v, t \ge 0$. To see this, suppose $u \ge v$ without loss of generality so that $\e^{-\frac{u^2}{2}} - \e^{-\frac{v^2}{2}} \le 0$ and
\begin{align*}
u \sinh(tu) \cosh(tv) - v \cosh(tu) \sinh(tv) &\ge v (\sinh(tu) \cosh(tv) - \cosh(tu) \sinh(tv))\\
& = v \sinh(tu - tv) \ge 0.
\end{align*}
This analysis further reveals that $f''(t) < 0$ for $t > 0$. Hence,
$f$ is strictly concave.
\end{proof}

To prove the other two lemmas, we collect some useful facts about  \eqref{eq:phase_cri}. This equation has (a) one positive root for $\delta \ge 1$ or $\delta < 1, \epsilon = \epsilon^\star$, (b) two positive roots for $\delta < 1$ and $\epsilon < \epsilon^\star$, and (c) no positive root if $\delta < 1$ and $\epsilon > \epsilon^\star$. In the case of (a) and (b), call $t(\epsilon, \delta)$ the positive root of \eqref{eq:phase_cri} (choose the larger one if there are two). Then $t(\epsilon, \delta)$ is a decreasing function of $\epsilon$. In particular, $t(\epsilon,\delta) \goto \infty$ as $\epsilon \goto 0$. In addition, $2(1-\epsilon)\left[ (1+t^2)\Phi(-t) - t\phi(t) \right] + \epsilon(1 + t^2) > \delta$ if $t > t(\epsilon, \delta)$.

\begin{lemma}\label{lm:root_exist}
For any $0 < u <  u^\star$, \eqref{eq:gamma_sandwich} has a unique root, denoted by $t^{\star}(u)$, in $(\alpha_0, \infty)$. In addition, $t^\star(u)$ strictly decreases as $u$ increases, and it further obeys $0 < (1 - u)/(1-2\Phi(-t^\star(u))) < 1$.
\end{lemma}

\begin{lemma}\label{lm:gamma_expr_3}
As a function of $u$,
\[
q^\star(u) = \frac{2(1-\epsilon)\Phi(-t^\star(u))}{2(1-\epsilon)\Phi(-t^\star(u)) + \epsilon u}
\]
is strictly increasing on $ (0,  u^\star)$.
\end{lemma}

\begin{proof}[Proof of Lemma~\ref{lm:gamma_max}]
We first consider the regime: $\delta < 1, \epsilon > \epsilon^\star$. By \eqref{eq:fdr_power_compact}, it is sufficient to show that $\fpower = 2(1-\epsilon')\Phi(-\alpha) + \epsilon' <  u^\star$ under the constraints \eqref{eq:two_eqn_comb}. 
From \eqref{eq:amp_restrict_w2} it follows that
\[
\Phi(-\alpha) = \frac12 \P(|W| > \alpha) < \frac{\delta - \epsilon\epsilon'}{2(1 - \epsilon\epsilon')},
\]
which can be rearranged as
\[
2(1-\epsilon')\Phi(-\alpha) + \epsilon' < \frac{(1-\epsilon')(\delta - \epsilon\epsilon')}{1 - \epsilon\epsilon'} + \epsilon'.
\]
The right-hand side is an increasing function of $\epsilon'$ because
its derivative is equal to
$(1-\epsilon)(1-\delta)/(1 - \epsilon\epsilon')^2$ and is
positive. Since the range of $\epsilon'$ is
$(0, \epsilon^\star/\epsilon)$, we get
\[
2(1-\epsilon')\Phi(-\alpha) + \epsilon'  < \frac{(1-\epsilon^\star/\epsilon)(\delta - \epsilon \cdot \epsilon^\star/\epsilon)}{1 - \epsilon \cdot \epsilon^\star/\epsilon} + \epsilon^\star/\epsilon =  u^\star.
\]
This bound $ u^\star$ can be arbitrarily approached: let
$\epsilon' = \epsilon^\star/\epsilon$ in the example given in
Section~\ref{sec:sharpness}; then set $\lambda = \sqrt{M}$ and take
$M \goto \infty$.


We turn our attention to the easier case where $\delta \ge 1$, or
$\delta < 1$ and $\epsilon \le \epsilon^\star$. By definition, the upper limit $ u^\star = 1$ trivially holds. It remains
to argue that $\fpower$ can be arbitrarily close to 1. To see this,
set $\Pi^\star = M$ almost surely, and take the same limits as before:
then $\fpower \goto 1$.

\end{proof}


\begin{proof}[Proof of Lemma~\ref{lm:alpha_large_gm}]

We begin by first considering the boundary case: 
\begin{equation}\label{eq:gamma_boundary}
\fpower = u.
\end{equation}
In view of \eqref{eq:fdr_power_compact}, we can write
\[
\ffdp = \frac{2(1-\epsilon)\Phi(-\alpha)}{2(1-\epsilon)\Phi(-\alpha) + \epsilon\ftpp} = \frac{2(1-\epsilon)\Phi(-\alpha)}{2(1-\epsilon)\Phi(-\alpha) + \epsilon u}.
\]
Therefore, a lower bound on $\ffdp$ is equivalent to maximizing $\alpha$ under the constraints \eqref{eq:two_eqn_comb} and \eqref{eq:gamma_boundary}.

Recall
$\E \eta_{\alpha}(W)^2 = 2(1+\alpha^2)\Phi(-\alpha) -
2\alpha\phi(\alpha)$.
Then from \eqref{eq:gamma_boundary} and \eqref{eq:amp_restrict_w} we
obtain a sandwiching expression for $1 - \epsilon'$:
\[
\frac{(1 - \epsilon) \left[ 2(1+\alpha^2)\Phi(-\alpha) - 2\alpha\phi(\alpha) \right] + \epsilon(1 + \alpha^2) - \delta}{\epsilon \left[ (1+\alpha^2)(1-2\Phi(-\alpha)) + 2\alpha\phi(\alpha) \right]} < 1-\epsilon' = \frac{1 - u}{1 - 2\Phi(-\alpha)},
\]
which implies
\[
\frac{(1 - \epsilon) \left[ 2(1+\alpha^2)\Phi(-\alpha) - 2\alpha\phi(\alpha) \right] + \epsilon(1 + \alpha^2) - \delta}{\epsilon \left[ (1+\alpha^2)(1-2\Phi(-\alpha)) + 2\alpha\phi(\alpha) \right]} - \frac{1 - u}{1 - 2\Phi(-\alpha)} < 0.
\]
The left-hand side of this display tends to
$1 - (1 - u) = u > 0$ as $\alpha \goto \infty$, and takes on the value 0 if $\alpha = t^\star(u)$. Hence, by
the uniqueness of $t^\star(u)$ provided by
Lemma~\ref{lm:root_exist}, we get $\alpha < t^\star(u)$. In 
conclusion, 
\begin{equation}\label{eq:fdp_low_alpha}
\ffdp = \frac{2(1-\epsilon)\Phi(-\alpha)}{2(1-\epsilon)\Phi(-\alpha) + \epsilon u} > \frac{2(1-\epsilon)\Phi(-t^\star(u))}{2(1-\epsilon)\Phi(-t^\star(u)) + \epsilon u} = q^\star(u).
\end{equation}
It is easy to see that $\ffdp$ can be arbitrarily close to $q^\star(u)$.

To finish the proof, we proceed to consider the general case $\fpower = u' > u$. The previous discussion clearly remains valid, and hence \eqref{eq:fdp_low_alpha} holds if $u$ is replaced by $u'$; that is, we have
\[
\ffdp > q^\star(u').
\]
By Lemma~\ref{lm:gamma_expr_3}, it follows from the monotonicity of $q^\star(\cdot)$ that $q^\star(u') > q^\star(u)$. Hence, $\ffdp > q^\star(u)$, as desired.


\end{proof}


\subsection{Proofs of auxiliary lemmas}
\label{sec:proofs-auxil-lemm-b}

\begin{proof}[Proof of Lemma~\ref{lm:root_exist}]
Set 
\[
\zeta := 1 - \frac{2(1 - \epsilon)\left[ (1+t^2)\Phi(-t) - t\phi(t) \right] + \epsilon(1 + t^2) - \delta}{\epsilon\left[ (1+t^2)(1-2\Phi(-t)) + 2t\phi(t) \right]}
\]
or, equivalently, 
\begin{equation}\label{eq:two_eps_eq}
2(1 - \epsilon\zeta) \left[ (1+t^2)\Phi(-t) - t\phi(t) \right] + \epsilon\zeta(1 + t^2) = \delta.
\end{equation}
As in Section~\ref{sec:proofs-lemm-reflm:g}, we abuse notation a
little and let $t(\zeta)=t(\epsilon\zeta, \delta)$ denote the (larger)
positive root of \eqref{eq:two_eps_eq}. Then the discussion about
\eqref{eq:phase_cri} in Section~\ref{sec:proofs-lemm-reflm:g} shows
that $t(\zeta)$ decreases as $\zeta$ increases. Note that in the case
where $\delta < 1$ and $\epsilon > \epsilon^\star(\delta)$, the range
of $\zeta$ in \eqref{eq:two_eps_eq} is assumed to be $(0,
\epsilon^\star/\epsilon)$, since otherwise \eqref{eq:two_eps_eq} does
not have a positive root (by convention, set $\epsilon^\star(\delta) =
1$ if $\delta > 1$).

Note that \eqref{eq:gamma_sandwich} is equivalent to
\begin{equation}\label{eq:gamma_expr_as}
u = 1 - \frac{2(1 - \epsilon)\left[ (1+t^2)\Phi(-t) - t\phi(t) \right] + \epsilon(1 + t^2) - \delta}{\epsilon\left[ (1+t^2)(1-2\Phi(-t)) + 2t\phi(t) \right]/(1 - 2\Phi(-t))}.
\end{equation}
Define
\[
h(\zeta) = 1 - \frac{2(1 - \epsilon)\left[ (1+t(\zeta)^2)\Phi(-t(\zeta)) - t(\zeta)\phi(t(\zeta)) \right] + \epsilon(1 + t(\zeta)^2) - \delta}{\epsilon\left[ (1+t(\zeta)^2)(1-2\Phi(-t(\zeta))) + 2t(\zeta)\phi(t(\zeta)) \right]/(1 - 2\Phi(-t(\zeta)))}.
\]
In view of \eqref{eq:two_eps_eq} and \eqref{eq:gamma_expr_as}, the proof of this lemma would be completed once we show the existence of $\zeta$ such that $h(\zeta) = u$. Now we prove this fact.

On the one hand, as $\zeta \searrow 0$, we see $t(\zeta)\nearrow \infty$. This leads to
\[
h(\zeta) \goto 0.
\]
On the other hand, if
$\zeta \nearrow \min\{1, \epsilon^\star/\epsilon\}$ , then
$t(\zeta)$ converges to $t^\star( u^\star) > \alpha_0$, which satisfies
\begin{equation}\nonumber
\begin{aligned}
2(1 - \min\{\epsilon,\epsilon^\star\}) \left[ (1+ t^{\star 2})\Phi(- t^\star) -  t^\star\phi( t^\star) \right] +  \min\{\epsilon,\epsilon^\star\}(1 +  t^{\star 2}) = \delta.
\end{aligned}  
\end{equation}
Consequently, we get
\[
h(\zeta) \goto  u^\star.
\]
Therefore, by the continuity of $h(\zeta)$, for any
$u \in (0,  u^\star)$ we can find
$0 < \epsilon' < \min\{ 1,\epsilon^\star/\epsilon \}$ such that $h(\epsilon') = u$. Put
$t^\star(u) = t(\epsilon')$. We have
\[
\frac{1 - u}{1 - 2\Phi(-t^\star(u))}  = 1 - \epsilon' < 1.
\]

Last, to prove the uniqueness of $t^\star(u)$ and its
monotonically decreasing dependence on $u$, it suffices to ensure
that (a) $t(\zeta)$ is a decreasing function of $\zeta$, and
(b) $h(\zeta)$ is an increasing function of $\zeta$. As seen
above, (a) is true, and (b) is also true as can be seen from writing
$h$ as
$h(\zeta) = 2(1-\zeta)\Phi(-t(\zeta)) + \zeta$, which
is an increasing function of $\zeta$.

\end{proof}


\begin{proof}[Proof of Lemma~\ref{lm:gamma_expr_3}]
Write
\[
q^\star(u) = \frac{2(1-\epsilon)}{2(1-\epsilon) + \epsilon u/\Phi(-t^{\star}(u))}.
\]
This suggests that the lemma amounts to saying that
$u/\Phi(-t^{\star}(u))$ is a decreasing function of $u$. From \eqref{eq:gamma_expr_as}, we see that this function is equal
to 
\[
\frac{1}{\Phi(-t^\star(u))}  - \frac{(1-2\Phi(-t^\star(u))) \left\{ 2(1 - \epsilon)\left[ (1+(t^\star(u))^2)\Phi(-t^\star(u)) - t^\star(u)\phi(t^\star(u)) \right] + \epsilon(1 + (t^\star(u))^2) - \delta \right\} } {\epsilon \Phi(-t^\star(u))\left[ (1+(t^\star(u))^2)(1-2\Phi(-t^\star(u))) + 2t^\star(u)\phi(t^\star(u)) \right]}.
\]
With the proviso that $t^\star(u)$ is decreasing in $u$,
it suffices to show that
\begin{multline}\nonumber
\frac{1}{\Phi(- t )} - \frac{(1-2\Phi(- t )) \left\{ 2(1 - \epsilon)\left[ (1+ t ^2)\Phi(- t ) -  t \phi( t ) \right] + \epsilon(1 +  t ^2) - \delta \right\} } {\epsilon \Phi(- t ) \left[ (1+ t ^2)(1-2\Phi(- t )) + 2 t \phi( t ) \right]} \\
= \frac{\delta}{\epsilon} \cdot \underbrace{\frac{1 - 2\Phi(- t )}{\Phi(- t ) \left[ (1+ t ^2)(1-2\Phi(- t )) + 2 t \phi( t ) \right]}}_{f_1( t )} - \frac2{\epsilon} \cdot \underbrace{ \frac{(1-2\Phi(- t )) \left[ (1+ t ^2)\Phi(- t ) -  t \phi( t ) \right]}{\Phi(- t )\left[ (1+ t ^2)(1-2\Phi(- t )) + 2 t \phi( t ) \right]} }_{f_2( t )} + 2
\end{multline}
is an increasing function of $ t > 0$. Simple calculations show that
$f_1$ is increasing while $f_2$ is  decreasing over $(0, \infty)$.
This finishes the proof.
\end{proof}





\section{Proof of Theorem 2.1}
\label{sec:proof-theor-refthm:m}

With the results given in Appendices~\ref{sec:uniform-over-lambda} and
\ref{sec:optimizing-tradeoff} in place, we are ready to characterize
the optimal false/true positive rate trade-off. Up until now, the
results hold for bounded $\lambda$, and we thus need to extend the
results to arbitrarily large $\lambda$. It is intuitively easy to
conceive that the support size of $\widehat{\bm\beta}$ will be small
with a very large $\lambda$, resulting in low power. The following
lemma, whose proof constitutes the subject of
Section~\ref{sec:proof-large-lambda}, formalizes this point. In this
section, $\sigma \ge 0$ may take on the value zero.  Also, we work
with $\lambda_0 = 0.01$ and $\eta = 0.001$ to carry fewer mathematical
symbols; any other numerical values would clearly work.

\begin{lemma}\label{lm:big_lamb}
For any $c > 0$, there exists $\lambda'$ such that
\[
\sup_{\lambda > \lambda'} \frac{ \# \{j: \widehat\beta_j(\lambda) \neq 0\}}{p} \le c
\]
holds with probability converging to one.
\end{lemma}

Assuming the conclusion of Lemma~\ref{lm:big_lamb}, we prove claim (b)
in Theorem 1 (noisy case), and then (a) (noiseless case). (c) is a
simple consequence of (a) and (b), and (d) follows from
Appendix~\ref{sec:optimizing-tradeoff}.

\paragraph{Case $\sigma > 0$}
Let $c$ be sufficiently small such that $q^\star(c/\epsilon) < 0.001$. Pick a large enough $\lambda'$ such that Lemma~\ref{lm:big_lamb} holds. Then with probability tending to one, for all $\lambda > \lambda'$, we have
\[
\power(\lambda) = \frac{T(\lambda)}{k \vee 1} \le (1 + o_{\P}(1)) \frac{\# \{j: \widehat\beta_j(\lambda) \neq 0\}}{\epsilon p} \le (1 + o_{\P}(1)) \frac{c}{\epsilon}.
\]
On this event, we get
\[
q^\star(\power(\lambda)) - 0.001 \le q^\star(c/\epsilon + o_{\P}(1) ) - 0.001 \le 0,
\]
which implies that
\begin{equation}\label{eq:bound_large_lamb}
\bigcap_{\lambda > \lambda'} \Big\{ \fdp(\lambda) \ge \mfdr\left( \power(\lambda) \right) - 0.001 \Big\}
\end{equation}
holds with probability approaching one.

Now we turn to work on the range $[0.01, \lambda']$. By
Lemma~\ref{lm:lamb_uniform}, we get that $V(\lambda)/p$
(resp.~$T(\lambda)/p$) converges in probability to $\fd(\lambda)$
(resp.~$\td(\lambda)$) uniformly over $[0.01, \lambda']$. As a
consequence,
\begin{equation}\label{eq:fdp_conver}
\fdp(\lambda) = \frac{V(\lambda)}{\max\left\{ V(\lambda) + T(\lambda),  1 \right\}} \plim \frac{\fd(\lambda)}{\fd(\lambda) + \td(\lambda) } = \ffdp(\lambda)
\end{equation}
uniformly over $\lambda \in [0.01, \lambda']$. The same reasoning also justifies that
\begin{equation}\label{eq:power_conver}
\power(\lambda) \plim  \fpower(\lambda)
\end{equation}
uniformly over $\lambda \in [0.01, \lambda']$. From Lemma~\ref{lm:alpha_large_gm} it follows that
\[
\ffdp(\lambda) > q^\star(\fpower(\lambda)).
\]
Hence, by the continuity of $q^\star(\cdot)$, combining
\eqref{eq:fdp_conver} with \eqref{eq:power_conver} gives that 
\[
\fdp(\lambda) \ge q^\star(\power(\lambda)) - 0.001
\]
holds simultaneously for all $\lambda \in [0.01, \lambda']$ with probability tending to one. This concludes the proof.

\paragraph{Case $\sigma = 0$}

Fix  $\lambda$ and let $\sigma > 0$ be sufficiently small. We first
prove that Lemma~\ref{lm:amp_theory} still holds for $\sigma = 0$ if
$\alpha$ and $\tau$ are taken to be the limiting solution to
\eqref{eq:amp_eqn} with $\sigma \goto 0$, denoted by $\alpha'$ and
$\tau'$. Introduce $\widehat\bbeta^\sigma$ to be the Lasso solution
with data $\by^\sigma := \bX \bbeta + \bz = \by + \bz$, where
$\bz \sim \mathcal{N}(\bm 0, \sigma^2 \bm I_n)$ is independent of $\bX$
and $\bbeta$. Our proof strategy is based on the approximate
equivalence between $\widehat\bbeta$ and $\widehat\bbeta^\sigma$.

It is well known that the Lasso residuals $\by - \bX\widehat \bbeta$ are
obtained by projecting the response $\by$ onto the polytope
$\{\bm r: \|\bX^\T \bm r\|_{\infty} \le \lambda \}$. The non-expansive
property of projections onto convex sets gives
\[
\left\| (\by^\sigma - \bX \widehat\bbeta^\sigma) - (\by - \bX \widehat\bbeta) \right\| \le \|\by^\sigma - \by\| = \|\bz\|.
\]
If $\bm P(\cdot)$ is the projection onto the polytope, then
$\bm I - \bm P$ is also non-expansive and, therefore, 
\begin{equation}\label{eq:Lasso_proj}
\left\| \bX \widehat\bbeta^\sigma - \bX \widehat\bbeta \right\| \le  \|\bz\|.
\end{equation}

Hence, from Lemma~\ref{lm:weak_corre} and $\|\bm z\| = (1 + o_{\P}(1)) \sigma\sqrt{n}$ it follows that, for any $c > 0$ and $r_c$ depending on $c$,
\begin{equation}\label{eq:corr_small}
\#\{1 \le j \le p: |\bX_j^\T(\by^\sigma - \bX \widehat\bbeta^\sigma - \by + \bX \widehat\bbeta)| > 2r_c \sigma \} \le cp
\end{equation}
holds with probability converging to one. Let $\bg$ and $\bg^\sigma$
be subgradients certifying the KKT conditions for $\widehat\bbeta$ and
$\widehat\bbeta^\sigma$. From
\begin{align*}
&\bX_j^\T (\by - \bX\widehat{\bbeta}) = \lambda g_j,\\
&\bX_j^\T (\by^\sigma - \bX\widehat{\bbeta}^\sigma) = \lambda g_j^\sigma,
\end{align*}
we get a simple relationship:
\[
\{ j : |g_j| \ge 1-a/2\} \setminus \{ j:  |g^\sigma_j | \ge 1 - a/2 - 2r_c \sigma/\lambda \} \subseteq \{j : |\bX_j^\T(\by^\sigma - \bX \widehat\bbeta^\sigma - \by + \bX \widehat\bbeta)| > 2r_c \sigma \}.  
\]
Choose $\sigma$ sufficiently small such that $2r_c \sigma/\lambda < a/2$, that is, $\sigma < a\lambda/(4 r_c )$. Then 
\begin{equation}\label{eq:support_inclusion}
\{ j : | g_j| \ge 1-a/2\} \setminus \{ j:  | g^\sigma_j | \ge 1 - a \} \subseteq \{j : |\bX_j^\T(\by^\sigma - \bX \widehat\bbeta^\sigma - \by + \bX \widehat\bbeta)| > 2r_c \sigma \}.
\end{equation}
As earlier, denote by $\mathcal{S} = \supp{\widehat\bbeta}$ and
$\mathcal{S}^\sigma = \supp{\widehat\bbeta^\sigma}$. In addition, let
$\mathcal{S}_v = \{ j : |g_j| \ge 1 - v\}$ and similarly
$\mathcal{S}_v^\sigma = \{ j : |g_j^\sigma| \ge 1 - v\}$. Notice
that we have dropped the dependence on $\lambda$ since $\lambda$ is
fixed. Continuing, since
$\mathcal{S} \subseteq \mathcal{S}_{\frac{a}{2}}$, from
\eqref{eq:support_inclusion} we obtain
\begin{equation}\label{eq:noiless_supp}
\mathcal{S} \setminus \mathcal{S}_a^\sigma\subseteq  \{j : |\bX_j^\T(\by^\sigma - \bX \widehat\bbeta^\sigma - \by + \bX \widehat\bbeta)| > 2r_c \sigma \}.
\end{equation}


This suggests that we apply Proposition 3.6 of \cite{BM12} that claims\footnote{Use a continuity argument to carry over the result of this proposition for finite $t$ to $\infty$.} the existence of positive constants $a_1 \in (0, 1), a_2$, and $a_3$ such that with probability tending to one, 
\begin{equation}\label{eq:cn_good}
\sigma_{\min}(\bX_{\mathcal{S}^\sigma_{a_1} \cup \mathcal{S}'}) \ge a_3
\end{equation}
for all $|\mathcal{S}'| \le a_2 p$. These constants also have positive limits $a_1', a_2', a_3'$, respectively, as $\sigma \goto 0$. We take $a < a'_1, c < a'_2$ (we will specify $a, c$ later) and sufficiently small $\sigma$ in \eqref{eq:noiless_supp}, and $\mathcal{S}' = \{j : |\bX_j^\T(\by^\sigma - \bX \widehat\bbeta^\sigma - \by + \bX \widehat\bbeta)| > 2r_c\sigma \}$. 
Hence, on this event, \eqref{eq:corr_small}, \eqref{eq:noiless_supp}, and \eqref{eq:cn_good} together give
\[
\|\bX \widehat\bbeta - \bX \widehat\bbeta^\sigma\| = \left\| \bX_{\mathcal{S}^\sigma_{a} \cup \mathcal{S}'}(\widehat\bbeta_{\mathcal{S}^\sigma_{a} \cup \mathcal{S}'} - \widehat\bbeta^\sigma_{\mathcal{S}^\sigma_{a} \cup \mathcal{S}'}) \right\| \ge a_3 \|\widehat\bbeta - \widehat\bbeta^\sigma\|
\]
for sufficiently small $\sigma$, which together with \eqref{eq:Lasso_proj} yields
\begin{equation}\label{eq:diff_small_sig}
\| \widehat\bbeta - \widehat\bbeta^\sigma \| \le \frac{(1+o_{\P}(1))\sigma\sqrt{n}}{a_3}. 
\end{equation}
Recall that the $\|\widehat\bbeta\|_0$ is the number of nonzero
entries in the vector $\widehat\bbeta$. From
\eqref{eq:diff_small_sig}, using the same argument outlined in
\eqref{eq:small_beta_even}, \eqref{eq:small_beta_big}, and
\eqref{eq:support_continu2}, we have
\begin{equation}\label{eq:noiless_supp2}
\|\widehat\bbeta\|_0 \ge \|\widehat\bbeta^\sigma\|_0 - c' p + o_{\P}(p)
\end{equation}
for some constant $c' > 0$ that decreases to $0$ as $\sigma/a_3 \goto 0$. 

We now develop a tight upper bound on $\|\widehat\bbeta\|_0$. Making
use of \eqref{eq:noiless_supp} gives
\[
\|\widehat\bbeta \|_0 \le \|\widehat\bbeta^\sigma\|_0  + \#\{j: 1-a \le | g^\sigma_j| < 1 \}  + cp.
\]
As in \eqref{eq:apply_continu_event1}, (3.21) of \cite{BM12} implies
that
\begin{equation}\nonumber
\# \left\{ j:  (1 - a) \le |g_j^\sigma| < 1 \right\}/p \plim  \P\left( (1-a)\alpha\tau \le |\Theta + \tau W| < \alpha\tau \right).
\end{equation}
Note that both $\alpha$ and $\tau$ depend on $\sigma$, and as $\sigma \goto 0$, $\alpha$ and $\tau$ converge to, respectively, $\alpha' > 0$ and $\tau' > 0$. Hence, we get
\begin{equation}\label{eq:noiless_suppx}
\|\widehat\bbeta\|_0 \le \|\widehat\bbeta^\sigma\|_0 + c'' p + o_{\P}(p)
\end{equation}
for some constant $c'' > 0$ that can be made arbitrarily small if $\sigma \goto 0$ by first taking $a$ and $c$ sufficiently small.

With some obvious notation, a combination of \eqref{eq:noiless_supp2}
and \eqref{eq:noiless_suppx} gives
\begin{equation}\label{eq:vt_approx}
|V - V^\sigma| \le c''' p, \quad |T - T^\sigma| \le c''' p, \
\end{equation}
for some constant $c''' = c'''_\sigma \goto 0$ as $\sigma \goto 0$. As $\sigma \goto 0$, observe the convergence, 
\[
\textnormal{fd}^{\mathsmaller\infty,\sigma} = 2(1-\epsilon)\Phi(-\alpha) \goto 2(1-\epsilon)\Phi(-\alpha') =  \textnormal{fd}^{\mathsmaller\infty, 0},
\]
and 
\[
\textnormal{td}^{\mathsmaller\infty,\sigma} = \epsilon\P(|\Pi^\star+ \tau W| > \alpha\tau) \goto \epsilon\P(|\Pi^\star+ \tau' W| > \alpha'\tau') = \textnormal{td}^{\mathsmaller\infty, 0}.
\]
By applying Lemma~\ref{lm:amp_theory} to $V^\sigma$ and $T^\sigma$ and making use of \eqref{eq:vt_approx}, the conclusions 
\[
\frac{V}{p} \plim \textnormal{fd}^{\mathsmaller\infty, 0} \quad \text{and} \quad  \frac{T}{p} \plim \textnormal{td}^{\mathsmaller\infty, 0}
\]
follow. 

Finally, the results for some fixed $\lambda$ can be carried over to a
bounded interval $[0.01, \lambda']$ in exactly the same way as in the
case where $\sigma > 0$. Indeed, the key ingredients,
namely, Lemmas~\ref{lm:lamb_uniform} and \ref{lm:l2_local_continu}
still hold. To extend the results to $\lambda > \lambda'$, we resort
to Lemma~\ref{lm:big_lamb}.

For a fixed a prior $\Pi$, our arguments immediately give an
  instance-specific trade-off. Let $q^\Pi(\cdot; \delta, \sigma)$ be
  the function defined as 
\[
q^\Pi \left(\textnormal{tpp}^{\mathsmaller\infty,
    \mathsmaller\sigma}(\lambda); \delta, \sigma \right) =
\textnormal{fdp}^{\mathsmaller\infty, \mathsmaller\sigma}(\lambda)
\]
for all $\lambda > 0$. It is worth pointing out that the sparsity
parameter $\epsilon$ is implied by $\Pi$ and that $q^\Pi$ depends on
$\Pi$ and $\sigma$ only through $\Pi/\sigma$ (if $\sigma = 0$, $q^\Pi$
is invariant by rescaling). By definition, we always have
\[
q^\Pi(u; \delta, \sigma) > q^\star(u)
\]
for any $u$ in the domain of $q^\Pi$. As is implied by the proof, it
is {impossible} to have a series of instances $\Pi$ such that
$q^\Pi(u)$ converges to $q^\star(u)$ at \textit{two} different
points. Now, we state the instance-specific version of Theorem
\ref{thm:min_fdr_given_power}.
\begin{theorem}\label{thm:instance_specific}
  Fix $\delta \in (0, \infty)$ and assume the working hypothesis.  In
  either the noiseless or noisy case and for any arbitrary small
  constants $\lambda_0 > 0$ and $\eta > 0$, the event
\[
\bigcap_{\lambda \ge \lambda_0} \Big\{ \fdp(\lambda) \ge q^\Pi\left(
  \power(\lambda) \right) - \eta \Big\}
\]
holds with probability tending to one. 
\end{theorem}

\subsection{Proof of Lemma~\ref{lm:big_lamb}}
\label{sec:proof-large-lambda}

Consider the KKT conditions restricted to $\mathcal{S}(\lambda)$:
\[
\bX_{\mathcal{S}(\lambda)}^{\transp}\left( \by - \bX_{\mathcal{S}(\lambda)} \widehat{\bm\beta}(\lambda) \right) = \lambda \bg(\lambda).
\]
Here we abuse the notation a bit by identifying both $\widehat\bbeta(\lambda)$ and $\bg(\lambda)$ as $|\mathcal{S}(\lambda)|$-dimensional vectors. As a consequence, we get
\begin{equation}\label{eq:Lasso_expr}
\widehat{\bm\beta}(\lambda) = (\bX_{\mathcal{S}(\lambda)}^{\transp} \bX_{\mathcal{S}(\lambda)})^{-1}(\bX_{\mathcal{S}(\lambda)}^{\transp} \by - \lambda \bg(\lambda)).
\end{equation}
Notice that $\bX_{\mathcal{S}(\lambda)}^{\transp} \bX_{\mathcal{S}(\lambda)}$ is invertible almost surely since the Lasso solution has at most $n$ nonzero components for \textit{generic} problems (see e.g. \cite{tibshirani2013Lasso}).
By definition, $\widehat{\bm\beta}(\lambda)$ obeys
\begin{equation}\label{eq:Lasso_def}
\frac12 \|\by - \bX_{\mathcal{S}(\lambda)} \widehat{\bm\beta}(\lambda) \|^2 + \lambda \|\widehat{\bm\beta}(\lambda) \|_1 \le \frac12 \|\by - \bX_{\mathcal{S}(\lambda)} \cdot \bm{0} \|^2 + \lambda \|\bm 0\|_1 = \frac12 \|\bm y\|^2.
\end{equation}
Substituting \eqref{eq:Lasso_expr} into \eqref{eq:Lasso_def} and applying the triangle inequality give
\begin{equation*}
\frac12 \left( \|\lambda \bX_{\mathcal{S}(\lambda)}(\bX_{\mathcal{S}(\lambda)}^{\transp} \bX_{\mathcal{S}(\lambda)})^{-1} \bg(\lambda)\| - \|\by -  \bX_{\mathcal{S}(\lambda)}(\bX_{\mathcal{S}(\lambda)}^{\transp} \bX_{\mathcal{S}(\lambda)})^{-1}\bX_{\mathcal{S}(\lambda)}^{\transp} \by\| \right)^2  + \lambda \|\widehat{\bm\beta}(\lambda) \|_1 \le \frac12 \|\bm y\|^2.
\end{equation*}
Since $\bm{I}_{|{\mathcal{S}(\lambda)}|} - \bX_{\mathcal{S}(\lambda)}(\bX_{\mathcal{S}(\lambda)}^{\transp} \bX_{\mathcal{S}(\lambda)})^{-1}\bX_{\mathcal{S}(\lambda)}^{\transp}$ is simply a projection, we get
\[
\|\by -  \bX_{\mathcal{S}(\lambda)}(\bX_{\mathcal{S}(\lambda)}^{\transp} \bX_{\mathcal{S}(\lambda)})^{-1}\bX_{\mathcal{S}(\lambda)}^{\transp} \by\|  \le \|\by\|.
\]
Combining the last displays gives
\begin{equation}\label{eq:Lasso_opt_conseq}
\lambda \|\bX_{\mathcal{S}(\lambda)}(\bX_{\mathcal{S}(\lambda)}^{\transp} \bX_{\mathcal{S}(\lambda)})^{-1} \bg(\lambda)\| \le 2\|\bm y\|,
\end{equation}
which is our key estimate.

Since $\sigma_{\max}(\bX) = (1 + o_{\P}(1))(1 +
1/\sqrt{\delta})$, 
\begin{equation}\label{eq:Lasso_supsize}
\lambda \|\bX_{\mathcal{S}(\lambda)}(\bX_{\mathcal{S}(\lambda)}^{\transp} \bX_{\mathcal{S}(\lambda)})^{-1} \bg(\lambda)\| \ge (1 + o_{\P}(1)) \frac{\lambda \|\bg(\lambda)\| }{1 + 1/\sqrt{\delta}} = (1 + o_{\P}(1)) \frac{\lambda \sqrt{|{\mathcal{S}(\lambda)}|}}{1 + 1/\sqrt{\delta}}.
\end{equation}
As for the right-hand side of \eqref{eq:Lasso_opt_conseq}, the law of large numbers reveals that
\[
2 \|\by\| = 2 \|\bX \bbeta + \bz\| = (2 + o_{\P}(1))\sqrt{n(\|\bbeta\|^2/n + \sigma^2)} = (2 + o_{\P}(1)) \sqrt{p\E\Theta^2 + n\sigma^2}.
\]
Combining the last two displays, it follows from
\eqref{eq:Lasso_opt_conseq} and \eqref{eq:Lasso_supsize} that
\[
\|\widehat{\bm\beta}(\lambda)\|_0 \equiv |{\mathcal{S}(\lambda)}_{\lambda}| \le (4 + o_{\P}(1)) \frac{(p \E\Theta^2 + n\sigma^2)(1+1/\sqrt{\delta})^2}{\lambda^2} = (1 + o_{\P}(1)) \frac{C p}{\lambda^2}
\]
for some constant $C$. It is worth emphasizing that the term $o_{\P}(1)$ is independent of any $\lambda > 0$. Hence, we finish the proof by choosing any $\lambda' > \sqrt{C/c}$.


\section{Proof of Theorem 3.1}
\label{sec:proof-theor-refthm:l}
We propose two preparatory lemmas regarding the $\chi^2$-distribution, which will be used in the proof of the theorem. 
\begin{lemma}\label{lm:gauss_concentration}
For any positive integer $d$ and $t \ge 0$, we have
\[
\P(\chi_d \ge \sqrt{d} + t) \le \e^{-t^2/2}.
\]
\end{lemma}

\begin{lemma}\label{lm:small_ball}
For any positive integer $d$ and $t \ge 0$, we have
\[
\P(\chi^2_d \le td) \le (\e t)^{\frac{d}{2}}.
\]
\end{lemma}

The first lemma can be derived by the Gaussian concentration inequality, also known as the Borell's inequality. The second lemma has a simple proof:
\begin{align*}
\P(\chi^2_d \le td) &= \int_0^{td} \frac1{2^{\frac{d}{2}}\Gamma(\frac{d}{2})}  x^{\frac{d}{2}-1} \e^{-\frac{x}{2}} \mathrm{d}x\\
&\le \int_0^{td} \frac1{2^{\frac{d}{2}}\Gamma(\frac{d}{2})}  x^{\frac{d}{2}-1} \mathrm{d} x =  \frac{2(td)^{\frac{d}{2}}}{d2^{\frac{d}{2}} \Gamma(\frac{d}{2})}.
\end{align*}
Next, Stirling's formula gives 
\[
\P(\chi^2_d \le td) \le  \frac{2(td)^{\frac{d}{2}}}{d2^{\frac{d}{2}} \Gamma(\frac{d}{2})} \le \frac{2(td)^{\frac{d}{2}}}{d2^{\frac{d}{2}} \sqrt{\pi d}(\frac{d}{2})^{\frac{d}{2}} \e^{-\frac{d}{2}}} \le (\e t)^{\frac{d}{2}}.
\]

Now, we turn to present the proof of Theorem~\ref{thm:l0}. Denote by $\mathcal{S}$ a subset of $\{1, 2, \ldots, p\}$, and let $m_0 = |\mathcal{S} \cap \{j: \beta_j = 0\}|$ and $m_1 = |\mathcal{S} \cap \{j: \beta_j \ne 0\}|$. Certainly, both $m_0$ and $m_1$ depend on $\mathcal{S}$, but the dependency is often omitted for the sake of simplicity. As earlier, denote by $k = \#\{j: \beta_j \ne 0\}$, which obeys $k = (\epsilon + o_{\P}(1))p$. Write $\widehat\bbeta^{\textnormal{LS}}_{\mathcal{S}}$ for the least-squares estimate obtained by regressing $\by$ onto $\bX_{\mathcal{S}}$. Observe that \eqref{eq:l0_ls} is equivalent to solving
\begin{equation}\label{eq:l0_ic}
\argmin_{\mathcal{S} \subset \{1, \ldots, p\}} \,  \| \by - \bX_{\mathcal{S}} \widehat\bbeta^{\textnormal LS}_{\mathcal{S}}\|^2 + \lambda |\mathcal{S}|.
\end{equation}
As is clear from \eqref{eq:l0_ls}, we only need to focus on
$\mathcal{S}$ with cardinality no more than $\min\{n, p\}$. Denote by
$\widehat{\mathcal{S}}$ the solution to \eqref{eq:l0_ic}, and define
$\widehat m_0$ and $\widehat m_1$ as before. To prove Theorem~\ref{thm:l0} it suffices to show the following: for
arbitrary small $c > 0$, we can find $\lambda$ and $M$ sufficiently
large such that \eqref{eq:l0_ls} gives
\begin{equation}\label{eq:thml0_key}
\P(\widehat m_0 > 2ck ~\mbox{or}~ \widehat m_1 \le (1-c)k ) \goto 0.
\end{equation}
Assume this is true. Then from \eqref{eq:thml0_key} we see that $\widehat m_0 \le 2ck$ and $\widehat m_1 > (1-c)k$ hold with probability tending to one. On this event, the TPP is
\[
 \frac{\widehat m_1}{k} > 1-c,
\]
and the FDP is
\[
\frac{\widehat m_0}{\widehat m_0 + \widehat m_1} \le \frac{2ck}{2ck + (1-c)k} = \frac{2c}{1+c}.
\]
Hence, we can have arbitrarily small FDP and almost full power by
setting $c$ arbitrarily small.

It remains to prove \eqref{eq:thml0_key} with proper choices of
$\lambda$ and $M$. Since
\[
\{\widehat m_0 > 2ck ~\mbox{or}~ \widehat m_1 \le (1-c)k\} \subset \{\widehat m_0 + \widehat m_1 > (1+c)k\} \cup \{\widehat m_1 \le (1-c)k\},
\]
we only need to prove
\begin{equation}\label{eq:thml0_key1}
\P(\widehat m_1 \le (1-c)k) \goto 0
\end{equation}
and
\begin{equation}\label{eq:thml0_key2}
\P(\widehat m_0 + \widehat m_1 > (1+c)k) \goto 0.
\end{equation}

We first work on \eqref{eq:thml0_key1}. Write
\[
\by = \sum_{j \in \mathcal{S}, \beta_j = M} M \bX_j + \sum_{j \in \overline{\mathcal{S}}, \beta_j = M} M \bX_j + \bz.
\]
In this decomposition, the summand
$\sum_{j \in \mathcal{S}, \beta_j = M} M \bX_j$ is already in the span
of $\bX_{\mathcal{S}}$. This fact implies that the residual vector
$\by - \bX_{\mathcal{S}} \widehat\bbeta^{\textnormal
  LS}_{\mathcal{S}}$
is the same as the projection of
$\sum_{j \in \overline{\mathcal{S}}, \beta_j = M} M \bX_j + \bz$ onto
the orthogonal complement of $\bX_{\mathcal{S}}$. Thanks to the
independence among $\bbeta, \bX$ and $\bz$, our discussion proceeds by
conditioning on the random support set of $\bbeta$. A crucial but
simple observation is that the orthogonal complement of
$\bX_{\mathcal{S}}$ of dimension $n - m_0 - m_1$ has uniform
orientation, independent of
$\sum_{j \in \overline{\mathcal{S}}, \beta_j = M} M \bX_j + \bz$. From
this fact it follows that
\begin{equation}\label{eq:dist_min}
L(\mathcal{S}) :=  \| \by - \bX_{\mathcal{S}} \widehat\bbeta^{\textnormal LS}_{\mathcal{S}}\|^2 + \lambda |\mathcal{S}| \overset{d}{=}  \left( \sigma^2 + M^2(k-m_1)/n \right) \chi^2_{n-m_0-m_1} + \lambda (m_0 + m_1).
\end{equation}
Call $E_{\mathcal{S},u}$ the event on which 
\[
L(\mathcal{S}) \le \sigma^2 (n-k+2u\sqrt{n-k}+u^2) + \lambda k
\]
holds; here, $u > 0$ is a constant to be specified later. In the
special case where $\mathcal{S} = \mathcal{T}$ and 
$\mathcal{T} \equiv \{j: \beta_j \ne 0\}$ is the true support,
Lemma~\ref{lm:gauss_concentration} says that this event has
probability bounded as 
\begin{equation}\label{eq:good_truesol}
\begin{aligned}
\P(E_{\mathcal{T},u}) &= \P\left( \sigma^2 \chi^2_{n-k} + \lambda k \le \sigma^2 (n-k+2u\sqrt{n-k}+u^2) + \lambda k \right) \\
&= \P\left( \chi^2_{n-k}  \le n-k+2u\sqrt{n-k}+u^2 \right)  \\ & 
\ge 1 - \e^{-\frac{u^2}{2}}.
\end{aligned}
\end{equation}
By definition, $E_{\widehat{\mathcal S},u}$ is implied by $E_{\mathcal{T}, u}$. Using this fact, we will show that $\widehat m_1$ is very close to $k$, thus validating \eqref{eq:thml0_key1}. By making use of 
\[
\{\widehat m_1 \le (1-c)k\} \subset \{\widehat m_0 + \widehat m_1 \ge (k+n)/2\} \cup \{\widehat m_1 \le (1-c)k, \widehat m_0 + \widehat m_1 < (k+n)/2\},
\]
we see that it suffices to establish that 
\begin{equation}\label{eq:key1_part1}
\P(\widehat m_0 + \widehat m_1 \ge (k+n)/2) \goto 0
\end{equation}
and
\begin{equation}\label{eq:key1_part2}
\P(\widehat m_1 \le (1-c)k, \widehat m_0 + \widehat m_1 < (k+n)/2) \goto 0
\end{equation}
for some $\lambda$ and sufficient large $M$. For \eqref{eq:key1_part1}, we have
\begin{equation}\label{eq:event_decomp_kn}
\begin{aligned}
\P(\widehat m_0 + \widehat m_1 \ge (k+n)/2) &\le \P(\overline E_{\mathcal T, u}) + \P(E_{\mathcal{T},u} \cap \{\widehat m_0 + \widehat m_1 \ge (k+n)/2\})\\
& \le \P(\overline E_{\mathcal T, u}) + \P(E_{\widehat{\mathcal{S}},u} \cap \{\widehat m_0 + \widehat m_1 \ge (k+n)/2\})\\
& \le \P(\overline E_{\mathcal T, u}) + \sum_{m_0 + m_1 \ge (k+n)/2} \P(E_{\mathcal{S},u}) \\
& \le \e^{-\frac{u^2}{2}}+ \sum_{m_0 + m_1 \ge (k+n)/2} \P(E_{\mathcal{S},u}),
\end{aligned}
\end{equation}
where the last step makes use of \eqref{eq:good_truesol}, and the
summation is over all $\mathcal{S}$ such that
$m_0(\mathcal{S}) + m_1(\mathcal{S}) \ge (k+n)/2$. Due to
\eqref{eq:dist_min}, the event $\E_{\mathcal{S},u}$ has the same
probability as
\begin{equation}\label{eq:event_E}
\begin{aligned}
&\left( \sigma^2 + M^2(k-m_1)/n \right) \chi^2_{n-m_0-m_1} + \lambda (m_0 + m_1) \le \sigma^2 (n-k+2u\sqrt{n-k}+u^2) + \lambda k \\
& \Longleftrightarrow \chi^2_{n-m_0-m_1}   \le \frac{\sigma^2 (n-k+2u\sqrt{n-k}+u^2) + \lambda k - \lambda (m_0 + m_1)}{\sigma^2 + M^2(k-m_1)/n }.
\end{aligned}
\end{equation}
Since $m_0 + m_1 \ge (k+n)/2$, we get
\[
\sigma^2 (n-k+2u\sqrt{n-k}+u^2) + \lambda k - \lambda (m_0 + m_1) \le \sigma^2 (n-k+2u\sqrt{n-k}+u^2) - \lambda (n-k)/2.
\]
Requiring
\begin{equation}\label{eq:lamb_cond1}
\lambda > 2\sigma^2,
\end{equation}
would yield $\sigma^2 (n-k+2u\sqrt{n-k}+u^2) - \lambda (n-k)/2 < 0$ for
sufficiently large $n$ (depending on $u$) as $n-k \goto \infty$.  In
this case, we have $\P(E_{\mathcal{S},u}) = 0$ whenever
$m_0 + m_1 \ge (k+n)/2$. Thus, taking $u \goto \infty$ in
\eqref{eq:event_decomp_kn} establishes \eqref{eq:key1_part1}.

Now we turn to \eqref{eq:key1_part2}. Observe that
\begin{equation}\label{eq:two_joint_event}
\begin{aligned}
 & \P\left(\widehat m_1 \le (1-c)k ~ \mbox{and}  ~ \widehat m_0 + \widehat m_1 < (k+n)/2 \right) \\
  & \qquad \qquad \qquad \qquad \le \P(\overline E_{\mathcal{T},u}) + \P(E_{\mathcal{T},u} \cap \{\widehat m_1 \le (1-c)k ~ \mbox{and} ~ \widehat m_0 + \widehat m_1 < (k+n)/2 \})\\
  & \qquad \qquad \qquad \qquad \le \e^{-\frac{u^2}{2}} +  \P(E_{\widehat{\mathcal{S}},u} \cap \{\widehat m_1 \le (1-c)k ~ \mbox{and} ~ \widehat m_0 + \widehat m_1 < (k+n)/2 \})\\
  & \qquad \qquad \qquad \qquad \le \e^{-\frac{u^2}{2}} +\sum_{m_0 + m_1 < \frac{k+n}{2}, m_1 \le (1-c)k} \P(E_{\mathcal{S},u}). \\
\end{aligned}
\end{equation}
For $m_0 + m_1 < (k+n)/2$ and $m_1 \le (1-c)k$, notice that $n-m_0-m_1 > (n-k)/2 = (\delta -\epsilon + o_{\P}(1))p/2$, and $M^2(k-m_1)/n \ge cM^2k/n \sim (c\epsilon/\delta+o_{\P}(1)) M^2$.
Let $t_0>0$ be a constant obeying
\[
\frac{\delta-\epsilon}{5} (1 + \log t_0) + \log 2 < -1,
\]
then choose $M$ sufficiently large such that
\begin{equation}\label{eq:m_cond1}
\frac{2\sigma^2 (\delta -\epsilon) + 2\lambda \epsilon }{(\sigma^2 + c\epsilon M^2/\delta)(\delta-\epsilon) } < t_0.
\end{equation}
This gives
\[
\frac{\sigma^2 (n-k+2u\sqrt{n-k}+u^2) + \lambda k - \lambda (m_0 + m_1)}{(\sigma^2 + M^2(k-m_1)/n)(n-m_0-m_1) } < t_0 
\]
for sufficiently large $n$. Continuing \eqref{eq:two_joint_event} and applying Lemma~\ref{lm:small_ball}, we get
\begin{equation}\label{eq:rare_events}
\begin{aligned}
&\P\left(\widehat m_1 \le (1-c)k ~ \mbox{and} ~ \widehat m_0 + \widehat m_1 < (k+n)/2 \right) \\
&\qquad \qquad \qquad  \le \e^{-\frac{u^2}{2}} + \sum_{m_0 + m_1 < \frac{k+n}{2}, m_1 \le (1-c)k} \P(\chi^2_{n-m_0-m_1} \le  t_0 (n-m_0-m_1)) \\
&\qquad \qquad \qquad  \le \e^{-\frac{u^2}{2}} + \sum_{m_0 + m_1 < \frac{k+n}{2}, m_1 \le (1-c)k} (\e t_0)^{\frac{n-m_0-m_1}{2}}  \\
&\qquad \qquad \qquad  \le \sum_{m_0 + m_1 < (k+n)/2, m_1 \le (1-c)k} (\e t_0)^{\frac{(\delta-\epsilon)p}{5}}\\
&\qquad \qquad \qquad \le \e^{-\frac{u^2}{2}} +  2^p (\e t_0)^{\frac{(\delta-\epsilon)p}{5}}\\
& \qquad \qquad \qquad  \le \e^{-\frac{u^2}{2}} + \e^{-p}.
\end{aligned}
\end{equation}
Taking $u \goto \infty$ proves \eqref{eq:key1_part2}.

Having established \eqref{eq:thml0_key1}, we proceed to prove \eqref{eq:thml0_key2}. By definition,
\[
\begin{aligned}
 \| \by - \bX \widehat \bbeta\|^2 + \lambda \|\widehat \bbeta\|_0 & =  \| \by - \bX_{\widehat{\mathcal S}} \widehat \bbeta_{\widehat{\mathcal S}}^{\textnormal LS}\|^2 + \lambda \|\bbeta_{\widehat{\mathcal S}}^{\textnormal LS}\|_0\\
&\le  \| \by - \bX \bbeta\|^2 + \lambda \|\bbeta\|_0 \\
& =  \|\bz\|^2 + \lambda k.
\end{aligned}
\]
If
\begin{equation}\label{eq:lamd_cond_last}
\lambda > \frac{\sigma^2 \delta}{c\epsilon},
\end{equation}
then
\[
\widehat m_0 + \widehat m_1  \le \frac{\|\bz\|^2}{\lambda} + k = (1+o_{\P}(1))\frac{\sigma^2 n}{\lambda} + k \le (1 +c)k
\]
holds with probability tending to one, whence \eqref{eq:thml0_key2}.

To recapitulate, selecting $\lambda$ obeying \eqref{eq:lamb_cond1} and
\eqref{eq:lamd_cond_last}, and $M$ sufficiently large such that
\eqref{eq:m_cond1} holds, imply that \eqref{eq:thml0_key} holds. The proof of the theorem is complete.

\end{document}